\documentclass[10pt,journal,twocolumn,twoside]{autart} 
\usepackage[utf8]{inputenc}

\usepackage{color}
\usepackage{graphics} 
\usepackage{epsfig} 
\usepackage{amssymb} 
\usepackage{amsmath} 
\usepackage{algpseudocode}
\usepackage{algorithm}
\usepackage{epstopdf}
\usepackage{comment}
\usepackage{verbatim} 
 
\usepackage{booktabs} 

\usepackage{setspace}

\usepackage{tikz}
 \usetikzlibrary{arrows,patterns,mindmap,backgrounds,shadows}
\usetikzlibrary{backgrounds}
\usetikzlibrary{shapes,arrows,chains}
  
\usepackage{enumitem}
\def\slack{\mathrm{slack}}
\def\nom{\mathrm{nom}}
\def\tube{\mathrm{robust}}
\def\impl{\mathrm{impl}}
\def\Input{\mathrm{input}}
\def\tubeSlack{\mathrm{robust,soft}}
\usepackage{xspace}
\def\SoftP{Soft-\textbf{P}\xspace}
\def\SoftG{Soft-\textbf{G}\xspace}
\def\SoftT{Soft-\textbf{T}\xspace}
\usepackage{cite}

\usepackage{etoolbox}
\newbool{arxiv}
\setbool{arxiv}{true} 
\def\change{\textcolor{black}}
\def\changeNew{\textcolor{black}}

\newtheorem{assumption}{\bf Assumption}
\newtheorem{definition}{\bf Definition}
\newtheorem{theorem}{\bf Theorem}

\newtheorem{lemma}{\bf Lemma}

\newtheorem{remark}{\bf Remark}

\usepackage{url}

\begin{document}
\begin{frontmatter}

\title{A model predictive control framework with robust stability guarantees under unbounded disturbances}

\author{Johannes K\"ohler}$^{1,2}$\ead{j.kohler@imperial.ac.uk},
\author{ Melanie N. Zeilinger}$^1$\ead{mzeilinger@ethz.ch}          
\address{$^1$Institute for Dynamic Systems and Control, ETH Zurich, Zurich CH-8092, Switzerland}  
\address{$^2$Mechanical Engineering Department, Imperial College London, SW7 2AZ London, UK}  
\begin{abstract} 
To address feasibility issues in model predictive control (MPC), most implementations relax state constraints by using slack variables and adding a penalty to the cost. We propose an alternative strategy: relaxing the initial state constraint with a penalty. 
Compared to state-of-the-art soft constrained MPC formulations, the proposed formulation has two key features:
(i) input-to-state stability and bounds on the cumulative constraint violation for unbounded disturbances; 
(ii) close-to-optimal performance under nominal operating conditions. 
The idea is initially presented for open-loop asymptotically stable nonlinear systems by designing the penalty as a Lyapunov function, but we also show how to relax this condition to: i) Lyapunov stable systems; ii) stabilizable systems; and iii) utilizing an implicit characterization of the Lyapunov function.  
In the special case of linear systems, the proposed MPC formulation reduces to a quadratic program, and the offline design and online computational complexity are only marginally increased compared to a nominal design. 
Numerical examples demonstrate benefits compared to state-of-the-art soft-constrained MPC formulations.
\end{abstract}
\end{frontmatter} 
\section{Introduction}
Model predictive control (MPC) is an optimization-based control technique that can ensure  asymptotic stability, near-optimality, and satisfaction of state and input constraints for nonlinear systems~\cite{rawlings2017model,grune2017nonlinear}. 
These properties typically rely {on} constrained optimization problems, which can lead to a loss of feasibility during run-time, e.g., due to unexpected disturbances. As a result, practical application requires a separate mechanism to handle infeasibility during online operation~\cite{vada2001linear}. 
Constraints are often relaxed and replaced by a suitable penalty on constraint violation, which is commonly referred to as \textit{soft-constrained MPC}~\cite{de1994constraint,zheng1995stability,scokaert1999feasibility,kerrigan2000soft}.
However, even if soft state constraints are considered, the presence of the terminal set constraint can still result in feasibility issues and a naive softening of this constraint can yield stability issues. 
We propose an MPC formulation that ensures recursive feasibility, robust stability, \changeNew{and bounds on the constraint violation under} unbounded disturbances while retaining close-to-optimal performance and constraint satisfaction under nominal conditions.

\textit{Related work:} 
A standard MPC design for linear systems uses the linear quadratic regulator (LQR) to define a terminal penalty and an invariant terminal set~\cite{rawlings2017model}. 
This design yields infinite-horizon optimal performance when choosing a sufficiently large prediction horizon, depending on the initial condition~\cite{sznaier1987suboptimal,chmielewski1996constrained,scokaert1998constrained}. 
The feasible set can be increased by relaxing the state constraints using soft constraints and partially relaxing the terminal set constraint~\cite{zeilinger2014soft,wabersich2021soft,rakovic2021model,chatzikiriakos2024learning}. 
Similarly, for nonlinear systems, the linearization can be used to compute a terminal penalty using the LQR~\cite{chen1998quasi}, resulting in locally approximately optimal controllers~\cite[Thm. 6.4]{grune2008infinite}. 
Even if state constraints are relaxed using penalties, these formulations only ensure robust stability for sufficiently small disturbances~\cite{manzano2020robust,yu2014inherent,allan2017inherent}. 
Overall, while these approaches can achieve close-to-optimal performance in nominal conditions, closed-loop guarantees only hold for certain initial conditions and small disturbances, which can lead to feasibility issues during operation.

For a certain class of nonlinear systems, MPC formulations without terminal constraints can ensure semi-global asymptotic stability with a sufficiently large prediction horizon~\cite[Chap. 6]{grune2017nonlinear}.
However, these MPC formulations only approach optimal performance if the prediction horizon tends to infinity. 
Furthermore, the presence of state constraints leads to additional complications. 

For linear stable systems without state constraints, a globally stabilizing MPC can be designed using the infinite-horizon open-loop cost as a terminal penalty~\cite[Thm.~1]{rawlings1993stability}. 
\ifbool{arxiv}{To account for state constraints, \cite{zheng1995stability} penalizes the peak constraint violation over the infinite-horizon, which preserves global stability. 
In \cite{scokaert1999feasibility}, standard point-wise in time constraint violations are penalized, but the corresponding infinite-horizon terminal penalty increases the computational complexity. 
Using the fact that the penalty on constraint violation is quadratically bounded, a quadratic terminal penalty ensuring stability can be constructed \cite{feller2016relaxed}, see also \cite{oancea2023relaxed}. These MPC formulations ensure robust stability for arbitrary initial conditions and disturbances. However, they also lead to severe performance degradation during nominal operation due to the conservative terminal penalties based on open-loop stability.}{
This approach has been extended to also consider soft state constraints using different penalties~\cite{zheng1995stability,scokaert1999feasibility,feller2016relaxed,oancea2023relaxed}. However, these approaches all  lead to severe performance degradation during nominal operation due to the conservative terminal penalties based on open-loop stability.}

Loss of feasibility can also be avoided by using \textit{robust} MPC designs~\cite{kouvaritakis2016model,houska2018robust}, 
\changeNew{which guarantee constraint satisfaction for a given set of bounded disturbances.}  
 \emph{Stochastic} approaches can deal with unbounded disturbances by relying on probabilistic assumptions~\cite{mesbah2018stochastic}. 
Both robust and stochastic MPC formulations can lead to performance deteriorations under nominal operation. 
The presented methodology has some relation to recent (stochastic) MPC formulations that use a relaxed interpolating initial state constraint~\cite{schluter2022stochastic,kohler2022recursively}. 
However, the proposed penalty-based relaxation differs by providing nominal MPC performance guarantees and inherent robustness to unbounded disturbances\changeNew{, see also Remark~\ref{rk:compare_RMPC} for a qualitative comparison to robust MPC schemes}.

Overall, the authors are not aware of any finite-horizon MPC framework that ensures both: robust stability under {unbounded} disturbance and near-optimal performance, even for linear stable systems without state constraints. 

\textit{Contribution:} 
In this paper, we propose to relax the initial state constraint in MPC by penalizing the difference between the measured state and a nominally optimized initial state using a weighted (incremental) Lyapunov function. 
{To address unbounded disturbances with compact input constraints, we focus on open-loop stable systems; however, the method applies equally to stabilizable systems when soft input constraints are used.} 
The proposed MPC formulation ensures the following properties: 
\begin{itemize}
\item input-to-state stability (ISS, cf.~\cite{jiang2001input}), even for {unbounded} disturbances; 
\item a suitable bound on the cumulative constraint violation;
\item close-to-optimal performance under nominal operating conditions. 
\end{itemize}
In contrast to existing soft-constrained MPC formulations, the proposed formulation uniquely combines inherent robustness to {unbounded} disturbances with the nominal optimality guarantees of MPC.

\textit{Outline:} 
We first introduce the problem setup (Sec.~\ref{sec:prelim}) and the proposed approach (Sec.~\ref{sec:proposed}). 
Then, we analyse the theoretical properties (Sec.~\ref{sec:theory}). 
Lastly, we provide a qualitative comparison to soft-constrained MPC formulations for linear systems (Sec.~\ref{sec:linear}) 
and  highlight quantitative benefits in numerical examples (Sec.~\ref{sec:num}). 

\ifbool{arxiv}{The main exposition in this paper assumes that the system is open-loop asymptotically stable and an incremental Lyapunov function is known. 
Appendix~\ref{app:stability} shows how this restriction can be relaxed to Lyapunov stable systems, stabilizable systems, and implicit characterizations of the Lyapunov function. 
Appendix~\ref{app:robust} provides a robustified design that ensures constraint satisfaction for a specified disturbance magnitude while maintaining feasibility and stability guarantees even under unbounded disturbances. 
Appendix~\ref{app:num} contains details concerning the numerical examples.
}{} 

\textit{Notation:}
The interior of a set $\mathbb{X}$ is denoted by $\mathrm{int}(\mathbb{X})$. 
The set of integers in the interval $[a,b]\subseteq\mathbb{R}$ are denoted by $\mathbb{I}_{[a,b]}$. 
By $\mathcal{K}$, we denote the set of continuous functions $\alpha:\mathbb{R}_{\geq 0}\rightarrow\mathbb{R}_{\geq 0}$ that are strictly increasing and satisfy $\alpha(0)=0$. Functions $\alpha\in\mathcal{K}$ that are also radially unbounded are class $\mathcal{K}_\infty$ functions.  
Functions $\delta:\mathbb{I}_{\geq 0}\rightarrow\mathbb{R}_{\geq 0}$ which are continuous, decreasing, and satisfy $\lim_{k\rightarrow\infty}\delta(k)=0$ are denoted by $\delta\in\mathcal{L}$. 
By $\mathcal{KL}$ we denote functions $\beta:\mathbb{R}_{\geq 0}\times\mathbb{I}_{\geq 0}\rightarrow\mathbb{R}_{\geq 0}$ with $\beta(\cdot,k)\in\mathcal{K}$ and $\beta(r,\cdot)\in\mathcal{L}$ for any fixed $k\in\mathbb{I}_{\geq 0}$, $r\in\mathbb{R}_{\geq 0}$. 
The infinity norm of a vector $x\in\mathbb{R}^n$ is denoted by $\|x\|_\infty$. 
{For a positive definite matrix $Q=Q^\top\in\mathbb{R}^n$ and a vector $x\in\mathbb{R}^n$, we denote $\|x\|_Q:=\sqrt{x^\top Qx}$.} 
The maximal and minimal eigenvalue of a symmetric matrix $Q=Q^\top\in\mathbb{R}^{n\times n }$ {are} denoted by $\sigma_{\max}(Q)$ and $\sigma_{\min}(Q)$, respectively. 
The point-to-set distance for a set $\mathcal{A}\subset\mathbb{R}^n$ and a vector $x\in\mathbb{R}^n$ is denoted by $\|x\|_{\mathcal{A}}:=\inf_{s\in\mathcal{A}}\|x-s\|$.

\section{Problem setup} 
\label{sec:prelim} 
We consider the following nonlinear discrete-time system
\begin{align}
\label{eq:sys}
x(k+1)=f(x(k),u(k))+w(k),\quad x(0)=x_0,
\end{align}
with state $x(k)\in\mathbb{R}^n$, control input $u(k)\in\mathbb{U}\subseteq\mathbb{R}^m$, disturbances $w(k)\in\mathbb{R}^n$, time $k\in\mathbb{I}_{\geq 0}$, and initial state $x_0\in\mathbb{R}^n$. 
{The controller should  minimize the closed-loop performance measured by a stage cost $\ell:\mathbb{R}^n\times\mathbb{U}\rightarrow\mathbb{R}_{\geq 0}$ and drive the system to the origin.}
\begin{assumption}
\label{ass:standing} (Regularity conditions)
There exists a function $\alpha_\ell\in\mathcal{K}_\infty$, such that $\ell(x,u)\geq \alpha_\ell(\|x\|+\|u\|)$ for all $(x,u)\in\mathbb{R}^n\times\mathbb{U}$.
Furthermore, $f(0,0)=0$, $\ell(0,0)=0$, $0\in\mathrm{int}(\mathbb{U})$, $0\in\mathrm{int}(\mathbb{X})$, $f,\ell$ are continuous, and $\mathbb{U}$ is compact.
\end{assumption} 
In addition, a desired region of operation is encoded by a closed set $\mathbb{X}\subseteq\mathbb{R}^n$. The controller aims to satisfy \changeNew{the desired state constraints} $x(k)\in\mathbb{X}$\changeNew{, but violations may be unavoidable due to the unknown and unbounded disturbances $w(k)$. 
Hence, as is standard in soft-constrained MPC~\cite{de1994constraint,zheng1995stability,scokaert1999feasibility,kerrigan2000soft}, the controller aims to minimize the constraint violation $\|x(k)\|_{\mathbb{X}}$.}
In contrast to robust MPC, we do \emph{not} leverage a known bound on the disturbances in the design.
Instead, the proposed control design provides desirable closed-loop properties even if the disturbances are vanishingly small or extremely large, which is addressed via the concept of ISS.
\begin{definition}
\label{def:ISS}
(Input-to-state stability, ISS~\cite[Def.~3.1]{jiang2001input})
A system $x(k+1)=f_{\mathrm{cl}}(x(k),w(k))$ is ISS if there exist functions $\beta\in\mathcal{KL}$, $\gamma\in\mathcal{K}$, such that for any initial condition $x(0)\in\mathbb{R}^n$ and any disturbance sequence $w(k)\in\mathbb{R}^n$, $k\in\mathbb{I}_{\geq 0}$, and all times $k\in\mathbb{I}_{\geq 0}$ it holds that
\begin{align}
\label{eq:ISS_def}
\|x(k)\|\leq \beta(\|x(0)\|,k)+\gamma\left(\max_{j\in\mathbb{I}_{[0,k-1]}}\|w(j)\|\right).
\end{align}
\end{definition} 
%
\textit{Nominal MPC:} For a given initial state $x\in \mathbb{R}^n$ and input sequence $\mathbf{u}\in\mathbb{U}^N$, we denote the solution to~\eqref{eq:sys} with $w\equiv 0$ (nominal prediction) after $k$ steps by $x_{\mathbf{u}}(k,x)\in \mathbb{R}^n$, $k\in\mathbb{I}_{[0,N]}$ with $x_{\mathbf{u}}(0,x)=x$. Furthermore, for $\mathbf{u}\in\mathbb{U}^N$, $\mathbf{u}_k\in\mathbb{U}$ denotes the $k$-th element in the sequence with $k\in\mathbb{I}_{[0,N-1]}$. 
For a given prediction horizon $N\in\mathbb{I}_{\geq 1}$, we consider the following finite-horizon cost
\begin{align*}
\mathcal{J}_N(x,{\mathbf{u}}):=\sum_{k=0}^{N-1}\ell(x_{\mathbf{u}}(k,x),\mathbf{u}_k)+V_{\mathrm{f}}(x_{\mathbf{u}}(N,x)),
\end{align*}
where $V_{\mathrm{f}}:\mathbb{R}^n\rightarrow\mathbb{R}_{\geq 0}$ is a continuous terminal penalty. 
Given a measured state $x\in\mathbb{R}^n$, the nominal MPC is defined by the following optimization problem
\begin{align}
\label{eq:MPC_nom}
V_{\nom}(x)=&\min_{\mathbf{u}\in\mathbb{U}^N}\mathcal{J}_N(x,\mathbf{u})\\ 
\text{s.t. }&x_{\mathbf{u}}(k,x)\in\mathbb{X},~ k\in\mathbb{I}_{[0,N-1]},~x_{\mathbf{u}}(N,x)\in\mathbb{X}_{\mathrm{f}},\nonumber
\end{align}
with a (closed) terminal set $\mathbb{X}_{\mathrm{f}}\subseteq\mathbb{X}$. 
Problem~\eqref{eq:MPC_nom} admits a minimizer as $\mathcal{J}_N$ is continuous and $\mathbb{U}$ compact~\cite[Prop.~2.4]{rawlings2017model}. 
{Without loss of generality, we choose a unique minimizer, denoted by  $\mathbf{u}^*(x)\in\mathbb{U}^N$.} 
The corresponding control law is given by applying the first element of the open-loop optimal input sequence, i.e., $\pi_{\nom}(x):=\mathbf{u}^*_0(x)\in\mathbb{U}$. 
The set of states $x$ for which~\eqref{eq:MPC_nom} is feasible is the feasible set $\mathbb{X}_N\subseteq\mathbb{X}$. 
We consider the following standard conditions on the terminal penalty $V_{\mathrm{f}}$ and the terminal set $\mathbb{X}_{\mathrm{f}}$~\cite{rawlings2017model}. 
\begin{assumption} (Terminal ingredients)
\label{ass:terminal}
There exists a terminal control law $k_{\mathrm{f}}:\mathbb{X}_{\mathrm{f}}\rightarrow\mathbb{U}$ and a function $\alpha_{\mathrm{f}}\in\mathcal{K}_\infty$, such that for any $x\in\mathbb{X}_{\mathrm{f}}$, we have:
\begin{itemize}
\item Terminal penalty: 
$V_{\mathrm{f}}(x^+)\leq V_{\mathrm{f}}(x)-\ell(x,u)$;
\item Constraint satisfaction: $(x,u)\in\mathbb{X}\times\mathbb{U}$;
\item Positive invariance: $x^+\in\mathbb{X}_{\mathrm{f}}$;
\item Weak controllability: $V_{\mathrm{f}}(x)\leq \alpha_{\mathrm{f}}(\|x\|)$, $0\in\mathrm{int}(\mathbb{X}_{\mathrm{f}})$, $\mathbb{X}_N$ is compact;
\end{itemize}\vspace{-2mm}
with $u=k_{\mathrm{f}}(x)$ and $x^+=f(x,u)$. 
\end{assumption}
Standard design procedures to constructively satisfy these conditions can be found in~\cite{rawlings2017model,chen1998quasi}. 
The following lemma ensures that this nominal MPC provides suitable guarantees under nominal operating conditions.
\begin{lemma}
\label{lemma:MPC}\cite[Sec.~2.4.2]{rawlings2017model}
Let Assumptions~\ref{ass:standing} and \ref{ass:terminal} hold. 
There exist functions $\alpha_1,\alpha_2\in\mathcal{K}_\infty$, such that for any $x\in\mathbb{X}_N$ it holds that
\begin{subequations}
\label{eq:nom_MPC_Lyap}
\begin{align}
\label{eq:nom_MPC_Lyap_1}
\alpha_1(\|x\|)\leq& V_{\nom}(x)\leq \alpha_2(\|x\|),\\
\label{eq:nom_MPC_Lyap_2}
V_{\nom}(f(x,\pi_{\nom}(x))\leq &V_{\nom}(x)-\ell(x,\pi_{\nom}(x)).
\end{align}
\end{subequations}
Suppose further that $x_0\in\mathbb{X}_N$ and $w\equiv 0$. Then,  the closed-loop system~\eqref{eq:sys} with $u(k)=\pi_{\nom}(x(k))$ ensures that Problem~\eqref{eq:MPC_nom} is recursively feasible, the constraints are satisfied, i.e., $x(k)\in\mathbb{X}$, $u(k)\in\mathbb{U}$ $\forall k\in\mathbb{I}_{\geq 0}$, $x=0$ is asymptotically stable, and the following performance bound holds:
\begin{align}
\label{eq:nom_MPC_performance}
\sum_{k=0}^\infty\ell(x(k),u(k))\leq V_{\nom}(x_0).  
\end{align} 
\end{lemma}
%
\begin{remark}(Optimality of nominal MPC)
\label{rk:local_opt}
Inequality~\eqref{eq:nom_MPC_performance} characterizes the performance of nominal MPC. 
Under an appropriate design of $V_{\mathrm{f}},\mathbb{X}_{\mathrm{f}}$, this performance bound ensures that the nominal MPC is locally close-to-optimal, for details see~\cite[Thm.~8]{magni2001stabilizing}, \cite[Thm.~6.2/6.4]{grune2008infinite}, \cite[Thm.~5.22]{grune2017nonlinear}, \cite[App.~A]{kohler2021dynamic}.  
\end{remark} 
\section{Proposed formulation}
\label{sec:proposed}
Under nominal operating conditions, i.e., $x_0\in\mathbb{X}_N$ and $w\equiv 0$, the closed-loop properties provided by the nominal MPC in Lemma~\ref{lemma:MPC} (stability, performance, constraint satisfaction) are satisfactory. 
Our main goal is to ensure that comparable closed-loop properties are preserved  under {unbounded disturbances} $w(k)\neq 0$ and for general initial states $x_0\notin\mathbb{X}_N$. 
{Given a state $x\in\mathbb{R}^n$, the proposed MPC formulation determines a control input using the following optimization problem} 
\begin{align}
\label{eq:MPC_proposed}
V_{\slack}(x)=&\min_{\bar{x}\in\mathbb{R}^n,{\mathbf{u}}\in\mathbb{U}^N}\mathcal{J}_N(\bar{x},{{\mathbf{u}}})+\lambda V_\delta(x,\bar{x})\\
\text{s.t. }&x_{{\mathbf{u}}}(k,\bar{x})\in\mathbb{X},~ k\in\mathbb{I}_{[0,N-1]},~ x_{{\mathbf{u}}}(N,\bar{x})\in\mathbb{X}_{\mathrm{f}},\nonumber
\end{align}
where $\lambda>0$ is a tunable weight and $V_\delta$ is a suitable distance function specified later. 
Compared to the nominal MPC~\eqref{eq:MPC_nom}, we relax the initial state constraint $x=\bar{x}$ by optimizing over a nominal initial state $\bar{x}$ and adding a penalty. 
\changeNew{This relaxation effectively yields soft state constraints, which are key to ensuring continuity and inherent robustness.} 
Computational complexity is only mildly increased due to the additional decision variable $\bar{x}\in\mathbb{R}^n$ and the evaluation of the nonlinear function $V_\delta$. 
We denote a minimizing\footnote{%
A minimizer exists, since the cost is radially unbounded in both $\bar{x}$ and $\mathbf{u}$~\cite[Prop.~2.4]{rawlings2017model}, given Assumptions~\ref{ass:standing} and \ref{ass:increm_stab}.}
 nominal state and input sequence by $\bar{x}_{\slack}^*(x)\in\mathbb{X}_N$ and ${\mathbf{u}}_{\slack}^*(x)\in\mathbb{U}^N$. 
The resulting closed-loop system is given by
\begin{align}
\label{eq:closed_loop}
x(k+1)=&f(x(k),u(k))+w(k),\nonumber\\
u(k)=&\pi_{\slack}(x(k)):=\mathbf{u}^*_{\slack,0}({x(k)})\in\mathbb{U}. 
\end{align}
Given bounded control inputs $\mathbb{U}$ and {unbounded} disturbances $w$, robust stability properties can in general only be {obtained} for open-loop asymptotically stable systems.  
Hence, we restrict ourselves to open-loop stable systems, which are characterized by a known incremental Lyapunov function $V_\delta(x,z)$ ~\ifbool{arxiv}{\cite{angeli2002lyapunov,tran2019convergence}}{\cite{angeli2002lyapunov}}.
\begin{assumption}
\label{ass:increm_stab}
({Open-loop stability})
The  function $V_\delta:\mathbb{R}^n\times\mathbb{R}^n\rightarrow\mathbb{R}_{\geq 0}$ is continuous and there exist functions $\alpha_{\delta,1},\alpha_{\delta,2},\alpha_{\delta,3},\alpha_{\delta,4}\in\mathcal{K}_\infty$, such that for all 
$x,z,w\in\mathbb{R}^n$, and all $u\in\mathbb{U}$, we have
\begin{subequations}
\label{eq:increm_Lyap}
\begin{align}
\label{eq:increm_Lyap_1}
&\alpha_{\delta,1}(\|x-z\|)\leq V_{\delta}(x,z)\leq \alpha_{\delta,2}(\|x-z\|),\\
\label{eq:increm_Lyap_2}
&V_{\delta}(f(x,u)+w,f(z,u))-V_{\delta}(x,z)\nonumber\\
\leq &-\alpha_{\delta,3}(\|x-z\|)+\alpha_{\delta,4}(\|w\|). 
\end{align}
\end{subequations}
\end{assumption} 
Inequalities~\eqref{eq:increm_Lyap} with $w=0$ are an equivalent characterization of asymptotic incremental stability~\ifbool{arxiv}{\cite{angeli2002lyapunov,tran2019convergence}}{\cite{angeli2002lyapunov}}, i.e., different initial conditions subject to the same control input asymptotically converge to the same state trajectory. 
For $w\neq 0$, Condition~\eqref{eq:increm_Lyap_2} additionally ensures incremental ISS w.r.t. disturbances $w$ (cf.~\cite{bayer2013discrete,angeli2002lyapunov}).  

{\begin{remark}(Open-loop stability \& Lyapunov function)
The proposed design~\eqref{eq:MPC_proposed} requires a known Lyapunov function $V_\delta$ certifying open-loop exponential stability (Asm.~\ref{ass:increm_stab}).  
\changeNew{We adopt this formulation because it leads to a simpler presentation, while the considered assumptions can be relaxed in a straightforward manner, as discussed next.} 
\ifbool{arxiv}{In Appendix~\ref{app:Lyap_stable}, we show that semi-global asymptotic stability can still be ensured for Lyapunov stable systems in the absence of disturbances.} 
For unstable systems with unbounded disturbances, meaningful closed-loop properties can only be derived if the system is \textit{stabilizable} and input constraints $\mathbb{U}$ can be violated (cf.~\cite{feller2016relaxed,oancea2023relaxed}). 
In \ifbool{arxiv}{Appendix~\ref{app:input}}{\cite[A.3]{JK_Soft_extended}}, we show that the proposed design and theoretical analysis can be naturally generalized in this direction by \changeNew{using a pre-stabilizing feedback}.  
In addition, the design in Problem~\eqref{eq:MPC_proposed} requires a \textit{known} incremental Lyapunov function $V_\delta$. 
For linear systems, a quadratic function $V_\delta(x,\bar{x})=\|x-\bar{x}\|_P^2$ is trivially obtained using a Lyapunov equation (Sec.~\ref{sec:linear}). 
For nonlinear systems, quadratically bounded functions $V_\delta$ can be computed using \textit{contraction metrics}~\cite{manchester2017control,singh2023robust}. 
To ease implementation, \ifbool{arxiv}{Appendix~\ref{app:implicit}}{\cite[A.1]{JK_Soft_extended}} provides an implicit characterization of $V_\delta$ using finite-horizon predictions,  assuming \textit{exponential} stability.
\end{remark}}

{\begin{remark}(Practical guideline for weight $\lambda>0$)
\label{rk:weight_lambda}
Suppose we have a nominal MPC that performs well under typical operating conditions. 
Our goal is to maintain this performance, while also ensuring reliable operation in the rare case of extreme disturbances. 
In this case, we recommend choosing a sufficiently large penalty $\lambda>0$, such that the resulting policy is virtually indistinguishable from a nominal MPC on the feasible set $\mathbb{X}_N$. 
This guideline is informed by the theoretical guarantees on stability, constraint satisfaction, and performance in Theorems~\ref{thm:global}--\ref{thm:performance} in the next section. 
\end{remark}}

{Intuitively, Problem~\eqref{eq:MPC_proposed} introduces a relaxation $x-\bar{x}$ on the initial state constraint. 
The following theoretical analysis establishes that:
\begin{itemize}
\item  this relaxation decays to zero if the disturbances decay (cf. Thm.~\ref{thm:global});
\item on the feasible set $x\in\mathbb{X}_N$, the relaxation is arbitrarily small for sufficiently large penalties $\lambda>0$ (cf. Thm.~\ref{thm:performance}, Eq.~\eqref{eq:performance_telescopic}).
\end{itemize}}

\begin{remark}\changeNew{(Qualitative differences to robust MPC)
\label{rk:compare_RMPC}
Satisfaction of constraints $x(k)\in\mathbb{X}$ is often ensured using robust MPC schemes, assuming a known bound on the disturbances is available~\cite{kouvaritakis2016model,houska2018robust}. 
Hence, robust MPC schemes are particularly well suited for safety-critical applications where constraint violations cannot be tolerated. 
However, practical application can be challenging since the true disturbance magnitude may exceed the assumed bound, while choosing a large disturbance set $\mathbb{W}$ can result in conservative designs or even render the MPC problem infeasible. 
In contrast, the proposed approach prioritizes keeping the close-to-optimal performance of MPC in the absence of disturbances, does not make any assumptions on the disturbances, but sacrifices guaranteed constraint satisfaction through a relaxation. 
In Theorems~\ref{thm:global}--\ref{thm:performance}, we establish qualitative bounds showing that the closed-loop stability, performance, and constraint satisfaction guarantees of the nominal MPC are recovered up to additional terms that depend continuously on the disturbance energy.}
\end{remark}

\section{Closed-loop theoretical analysis:\\ISS, constraint satisfaction, and performance}
\label{sec:theory}
In the following, we provide the theoretical analysis. 
 First, Lemma~\ref{lemma:equiv_relaxed} shows that Problem~\eqref{eq:MPC_proposed} can be analysed as a weighted projection on the feasible set $\mathbb{X}_N$. 
Theorem~\ref{thm:global} ensures ISS. 
Then, suitable bounds for the closed-loop constraint violation and performance are derived in Theorems~\ref{thm:constraints} and \ref{thm:performance}. 
\subsection{Weighted projection}
\label{sec:theory_projection}
For the following analysis, we repeatedly utilize the fact that Problem~\eqref{eq:MPC_proposed} is equivalent to
\begin{align}
\label{eq:MPC_equ}
V_{\slack}(x)=&\min_{\bar{x}\in\mathbb{X}_{N}}  V_{\nom}(\bar{x})+\lambda V_\delta(x,\bar{x}).
\end{align}
\begin{lemma}
\label{lemma:equiv_relaxed}
For all $x\in\mathbb{R}^n$, it holds that:
\begin{itemize}
\item Problem~\eqref{eq:MPC_proposed} and Problem~\eqref{eq:MPC_equ}  have the same minimizer $\bar{x}_{\slack}^*(x)$ and the same minimum $V_{\slack}(x)$;
\item
 $\pi_{\slack}(x)=\pi_{\nom}(\bar{x}_{\slack}^*(x))$. 
\end{itemize}
\end{lemma}
\begin{pf}
The constraints in Problem~\eqref{eq:MPC_proposed} on ($\bar{x},\mathbf{u}$) are equivalent to the constraints in Problem~\eqref{eq:MPC_nom} and hence to $\bar{x}\in\mathbb{X}_N$. 
Furthermore, any minimizer $\mathbf{u}_{\slack}^*$ in Problem~\eqref{eq:MPC_proposed} is a minimizer in Problem~\eqref{eq:MPC_nom} with initial condition $x=\bar{x}$, i.e., $\mathbf{u}_{\slack}^*(x)=\mathbf{u}^*(\bar{x}_{\slack}^*(x))$. This implies  $\pi_{\slack}(x)=\pi_{\nom}(\bar{x}_{\slack}^*(x))$. 
Lastly, $\mathcal{J}_N(\bar{x}_{\slack}^*(x),{\mathbf{u}}_{\slack}^*(x))=V_{\nom}(\bar{x}_{\slack}^*(x))$ implies that the same cost is minimized, resulting in the same minimizer $\bar{x}_{\slack}^*(x)$ and minimum $V_{\slack}(x)$.  $\hfill\square$
\end{pf}
Problem~\eqref{eq:MPC_equ} can be viewed as a (weighted) projection of
the state $x$ on the feasible set $\mathbb{X}_N$, which simplifies the analysis of the proposed MPC formulation.  
\changeNew{The value function $V_{\nom}$ is in general discontinuous~\cite{grimm2004examples}. By relaxing the initial state constraint with a penalty, the proposed formulation ensures continuity of the optimal cost $V_{\slack}$, as established in Theorem~\ref{thm:global} below. This property is crucial for establishing the robust closed-loop properties developed in the following.}
 
\subsection{Input-to-state stability}
\label{sec:theory_1}
The following theorem characterizes the robust stability properties of the proposed MPC formulation.
\begin{theorem}
\label{thm:global}
Let Assumptions~\ref{ass:standing}, \ref{ass:terminal}, and \ref{ass:increm_stab} hold. 
Then, the closed-loop system~\eqref{eq:closed_loop} is ISS (Def.~\ref{def:ISS}) and $V_{\slack}$ is a continuous ISS Lyapunov function. 
\end{theorem}
\begin{pf}
We show that there exist functions $\tilde{\alpha}_1$, $\tilde{\alpha}_2$, $\tilde{\alpha}_3$, $\tilde{\alpha}_4\in\mathcal{K}_\infty$, such that for any $x,w\in\mathbb{R}^n$: 
\begin{subequations}
\label{eq:Lyap_joint}
\begin{align}
\label{eq:Lyap_joint_1} 
\tilde{\alpha}_{1}(\|x\|)\leq& V_{\slack}(x)\leq \tilde{\alpha}_2(\|x\|),\\
\label{eq:Lyap_joint_2}
&V_{\slack}(f(x,\pi_{\slack}(x))+w)-V_{\slack}(x)\nonumber\\
\leq &-\tilde{\alpha}_3(\|x\|)+\tilde{\alpha}_4(\|w\|).
\end{align}
\end{subequations}
\textbf{Part I:} 
Note that $0\in\mathbb{X}_{\mathrm{f}}$ and $V_{\mathrm{f}}(0)=0$ ensure $0\in\mathbb{X}_N$ and $V_{\nom}(0)=0$. 
Thus, for all $x\in\mathbb{R}^n$, $\bar{x}=0$ is feasible candidate solution to Problem~\eqref{eq:MPC_equ}, ensuring $V_{\slack}(x)\leq \lambda V_\delta(x,0)\leq \lambda \alpha_{2,\delta}(\|x\|)$. 
Thus, \eqref{eq:Lyap_joint_1} holds with $\tilde{\alpha}_2:=\lambda\cdot\alpha_{2,\delta}\in\mathcal{K}_\infty$. \\
\textbf{Part II:} 
Define $\tilde{\alpha}_1(c):=\min\{\alpha_1(c/2),\lambda\alpha_{1,\delta}(c/2)\}$, $c\geq 0$, which satisfies $\tilde{\alpha}_1\in\mathcal{K}_\infty$~\cite{kellett2014compendium}. 
Using the weak triangle inequality (cf.~\cite[Eq.~(8)]{kellett2014compendium})
\begin{align}
\label{eq:weak_triang}
\alpha(a+b)\leq \alpha(2a)+\alpha(2b), ~ \forall a,b\geq 0,~ \alpha\in\mathcal{K}_\infty,
\end{align}
the lower bound in Inequality~\eqref{eq:Lyap_joint_1} holds with
\begin{align*}
V_{\slack}(x)\stackrel{\eqref{eq:MPC_equ}}{=}&V_{\nom}(\bar{x}_{\slack}^*(x))+\lambda V_\delta(x,\bar{x}_{\slack}^*(x))\\
\stackrel{\eqref{eq:nom_MPC_Lyap_1},\eqref{eq:increm_Lyap_1}}\geq &\alpha_1(\|\bar{x}_{\slack}^*(x)\|)+ \lambda\alpha_{\delta,1}(\|x-\bar{x}_{\slack}^*(x)\|)\\
\geq &\tilde{\alpha}_1(2\|\bar{x}_{\slack}^*(x)\|)+\tilde{\alpha}_1(2\|x-\bar{x}_{\slack}^*(x)\|)\\
\stackrel{\eqref{eq:weak_triang}}{\geq} &\tilde{\alpha}_1(\|\bar{x}_{\slack}^*(x)\|+\|x-\bar{x}_{\slack}^*(x)\|)
\geq \tilde{\alpha}_1(\|x\|).
\end{align*} 
\textbf{Part III: } Abbreviate $\bar{x}=\bar{x}_{\slack}^*(x)$, $\bar{x}^+=f(\bar{x},u)$, $x^+=f(x,u)+w$, and $u=\pi_{\slack}(x)=\pi_{\nom}(\bar{x})$ (Lemma~\ref{lemma:equiv_relaxed}). 
We use $\bar{x}^+$ as a feasible candidate solution in~\eqref{eq:MPC_equ}, yielding 
\begin{align}
\label{eq:global_decrease_1}
&V_{\slack}(x^+)\stackrel{\eqref{eq:MPC_equ}}{\leq} V_{\nom}(\bar{x}^+)+\lambda V_\delta(x^+,\bar{x}^+)\\
\stackrel{\eqref{eq:nom_MPC_Lyap_2},\eqref{eq:increm_Lyap_2}}{\leq}&V_{\nom}(\bar{x})-\ell(\bar{x},\pi_{\nom}(\bar{x}))\nonumber\\
&+\lambda (V_\delta(x,\bar{x})-\alpha_{3,\delta}(\|\bar{x}-x\|)+\alpha_{4,\delta}(\|w\|))\nonumber\\
\stackrel{\text{Asm.~\ref{ass:standing}},\eqref{eq:MPC_equ}}{\leq}&V_{\slack}(x)+\lambda\alpha_{4,\delta}(\|w\|)\nonumber\\
&-\alpha_\ell(\|\bar{x}\|)-\lambda \alpha_{3,\delta}(\|\bar{x}-x\|).\nonumber
\end{align}
Choose $\tilde{\alpha}_3(c):=\min\{\alpha_\ell(c/2),\lambda\alpha_{3,\delta}(c/2)\}$, $c\geq 0$, $\tilde{\alpha}_3\in\mathcal{K}_\infty$ and $\tilde{\alpha}_4:=\lambda\alpha_{4,\delta}\in\mathcal{K}_\infty$. 
Inequality~\eqref{eq:Lyap_joint_2} holds with 
\begin{align*}
\alpha_\ell(\|\bar{x}\|)+\lambda \alpha_{3,\delta}(\|\bar{x}-x\|)
\geq& \tilde{\alpha}_3(2\|\bar{x}\|)+\tilde{\alpha}_3(2\|\bar{x}-x\|)\\
\stackrel{\eqref{eq:weak_triang}}{\geq} \tilde{\alpha}_3(\|\bar{x}\|+\|\bar{x}-x\|)\geq& \tilde{\alpha}_3(\|x\|).&
\end{align*}
\textbf{Part IV: }
For any two {states} $x,\tilde{x}\in\mathbb{R}^n$, it holds that
\begin{align*}
V_{\slack}(x)\stackrel{\eqref{eq:MPC_equ}}{\leq} V_{\nom}(\bar{x}_{\slack}^*(\tilde{x}))+\lambda V_\delta(x,\bar{x}_{\slack}^*(\tilde{x}))\\
\stackrel{\eqref{eq:MPC_equ}}{=}V_{\slack}(\tilde{x})+\lambda (V_\delta(x,\bar{x}_{\slack}^*(\tilde{x}))-V_\delta(\tilde{x},\bar{x}_{\slack}^*(\tilde{x})). 
\end{align*}
Continuity of $V_{\slack}$ follows with $V_\delta$ continuous and $\lambda$ finite.
Inequalities~\eqref{eq:Lyap_joint} with $V_{\slack}$ continuous ensure that the closed-loop system~\eqref{eq:closed_loop} is ISS~\cite[Lemma 3.5]{jiang2001input}.$\hfill\square$
\end{pf}
Even though the value function $V_{\nom}$ of the nominal MPC~\eqref{eq:MPC_nom} is, in general, discontinuous~\cite{grimm2004examples}, the proposed relaxation \changeNew{ensures continuity of $V_{\slack}$.}
\changeNew{Moreover,} $V_{\slack}$ is an ISS Lyapunov function~\cite{jiang2001input} for the closed-loop system resulting from the \changeNew{proposed} MPC, even under unbounded disturbances $w(k)$.
 In the absence of disturbances ($w\equiv 0$), ISS also implies globally asymptotically stable.

\subsection{Constraint satisfaction}
For the following analysis, we strengthen our assumptions to obtain simple bounds for the cumulative constraint violation and performance. 
\begin{assumption} (Simplifying assumptions)
\label{ass:exp_bounds}
\begin{enumerate}[label=\alph*)]
\item  $\alpha_i$, $i\in\{\ell,\mathrm{f}\}$ are quadratic, i.e., $\alpha_i(r)=c_i r^2$. 
\label{ass:exp_bound_MPC}
\item  $\ell(x,u)=\|x\|_Q^2+\|u\|_R^2$ with positive definite matrices $Q,R$. 
\label{ass:exp_bound_ell_quad}
\item $\alpha_{i,\delta}$, $i\in\mathbb{I}_{[1,4]}$ are quadratic, i.e., $\alpha_{i,\delta}(r)=c_{i,\delta}r^2$. 
\label{ass:exp_bounds_alpha_delta}
\end{enumerate}
\end{assumption}
Conditions~\ref{ass:exp_bound_MPC}--\ref{ass:exp_bound_ell_quad} are naturally satisfied by standard MPC designs~\cite{rawlings2017model,chen1998quasi}.  
Condition~\ref{ass:exp_bounds_alpha_delta} requires that $V_\delta$ is an \textit{exponential} incremental Lyapunov function, which is the case for most constructive designs (cf.~\cite{manchester2017control,singh2023robust} and \ifbool{arxiv}{App.~\ref{app:implicit}}{\cite[App.~A.1]{JK_Soft_extended}}).  
%
The following theorem provides a bound on the cumulative constraint violation. 
\begin{theorem}
\label{thm:constraints}
Let Assumptions~\ref{ass:standing}, \ref{ass:terminal}, \ref{ass:increm_stab}, and \ref{ass:exp_bounds} hold. 
There exist constants $c_{\mathbb{X},1}$, $c_{\mathbb{X},2}$, $c_{\mathbb{X},3}\geq 0$, such that for any initial conditions $x_0\in\mathbb{R}^n$, weighting $\lambda>0$, and any disturbances $w(k)\in\mathbb{R}^n$, $k\in\mathbb{I}_{\geq 0}$, {and any time $T\in\mathbb{I}_{\geq 0}$}, the closed-loop system~\eqref{eq:closed_loop} satisfies
\begin{align}
\label{eq:constraint_violation}
\sum_{k=0}^{{T-1}}\|x(k)\|_{\mathbb{X}}^2\leq c_{\mathbb{X},1}\|x_0\|_{\mathbb{X}_N}^2+\dfrac{c_{\mathbb{X},2}}{\lambda}+c_{\mathbb{X},3}\sum_{k=0}^{{T-1}} \|w(k)\|^2. 
\end{align}
\end{theorem}
\begin{pf}
Given that $\alpha_\ell,\alpha_{\mathrm{f}}$ are quadratic, Lemma~\ref{lemma:MPC} and \cite[Prop.~2.16]{rawlings2017model} also imply quadratic bounds on $V_\nom$, i.e., 
 $\alpha_1(r)=c_1r^2$, $\alpha_2(r)=c_2r^2$, with $c_1=c_\ell>0$ and $$c_2:=\max\left\{c_{\mathrm{f}},\dfrac{\max_{x\in\mathbb{X}_N, \mathbf{u}\in\mathbb{U}^N}\mathcal{J}_N(x,\mathbf{u})}{\inf_{x\notin \mathbb{X}_{\mathrm{f}}}\|x\|^2}\right\}\in(0,\infty).$$ 
{Thus, Inequalities~\eqref{eq:nom_MPC_Lyap} yield 
\begin{align*}
V_{\nom}(\bar{x}^+)\leq \rho_{\nom}V_{\nom}(\bar{x}),
\end{align*}
with  $\rho_{\nom}=1-c_\ell/c_2\in[0,1)$.}
{Similarly, $\alpha_{i,\delta}$ quadratic and~\eqref{eq:increm_Lyap} imply 
\begin{align*}
V_\delta(x^+,\bar{x}^+)\leq \rho_{\delta}V_{\delta}(x,\bar{x})+c_{4,\delta}\|w\|^2,
\end{align*}
with $\rho_{\delta}=1-c_{3,\delta}/c_{2,\delta}\in[0,1)$.} 
Applying the arguments from Inequality~\eqref{eq:global_decrease_1} yields
\begin{align}
\label{eq:Lyap_exp_ISS}
&V_{\slack}(x^+)\leq V_{\nom}(\bar{x}^+)+\lambda V_\delta(x^+,\bar{x}^+)\nonumber\\
\leq &\rho_{\nom}V_{\nom}(\bar{x})+\lambda \rho_{\delta}V_{\delta}(x,\bar{x})+c_{4,\delta}\lambda\|w\|^2\nonumber\\
\leq& \tilde{\rho}V_{\slack}(x)+c_{4,\delta}\lambda\|w\|^2,
\end{align}
with $\tilde{\rho}=\max\{\rho_\delta,\rho_{\nom}\}\in[0,1)$ independent of $\lambda>0$. 
The point-wise in time constraint violation satisfies 
\begin{align}
\label{eq:constraint_violation_1}
\|x\|_{\mathbb{X}}^2\leq \|x-\bar{x}\|^2\stackrel{\eqref{eq:increm_Lyap_1}}{\leq} \frac{V_\delta(x,\bar{x})}{c_{\delta,1}}\stackrel{\eqref{eq:MPC_equ}}{\leq} \frac{V_{\slack}(x)}{c_{\delta,1}\lambda}.
\end{align}
Using Inequality~\eqref{eq:Lyap_exp_ISS} in a telescopic sum yields
\begin{align*}
&V_{\slack}(x(k))
\leq \tilde{\rho}^k V_{\slack}(x_0)+c_{\delta,4}\lambda\sum_{j=0}^{k-1}\tilde{\rho}^{k-j-1}\|w(j)\|^2.
\end{align*}
Summing up the squared constraint violation yields
\begin{align}
\label{eq:constraint_violation_2}
&c_{\delta,1}\sum_{k=0}^{{T-1}}\|x(k)\|_{\mathbb{X}}^2\leq \dfrac{1}{\lambda}\sum_{k=0}^{{T-1}} V_{\slack}(x(k))\nonumber\\
\leq& \sum_{k=0}^{{T-1}} \dfrac{\tilde{\rho}^k}{\lambda}V_{\slack}(x_0)
+c_{\delta,4}\sum_{k=0}^{{T-1}}\sum_{j=0}^{k-1}\tilde{\rho}^{k-j-1}\|w(j)\|^2\nonumber\\
{\leq}&\dfrac{V_{\slack}(x_0)}{\lambda(1-\tilde{\rho})}+\dfrac{c_{4,\delta}}{1-\tilde{\rho}}\sum_{j=0}^{{T-1}}\|w(j)\|^2.
\end{align}
Note that $\mathcal{J}_N$ continuous and $\mathbb{X}_N$, $\mathbb{U}$ compact imply a uniform upper bound $\bar{V}>0$, i.e., $V_{\nom}(\bar{x})\leq \bar{V}$ for any $\bar{x}\in\mathbb{X}_N$. 
Consider a candidate $\bar{x}\in\arg\min_{\bar{x}\in \mathbb{X}_N}\|x_0-\bar{x}\|$, which satisfies
$V_{\delta}(x_0,\bar{x})\stackrel{\eqref{eq:increm_Lyap_1}}{\leq} c_{\delta,2}\|x_0-\bar{x}\|^2= c_{\delta,2}\|x_0\|_{\mathbb{X}_N}^2$ and hence 
\begin{align*}
V_{\slack}(x_0)\leq \lambda V_{\delta}(x_0,\bar{x})+V_{\nom}(\bar{x})\leq c_{\delta,2}\lambda\|x_0\|^2_{\mathbb{X}_N}+\bar{V}.
\end{align*}
Inequality~\eqref{eq:constraint_violation} follows from~\eqref{eq:constraint_violation_2} by choosing\\\mbox{$c_{\mathbb{X},1}:=\frac{c_{\delta,2}}{(1-\tilde{\rho})c_{\delta,1}}$}, $c_{\mathbb{X},2}:=\frac{\bar{V}}{(1-\tilde{\rho})c_{\delta,1}}$, $c_{\mathbb{X},3}:=\frac{c_{4,\delta}}{(1-\tilde{\rho})c_{\delta,1}}$. $\hfill\square$
\end{pf} 
In the nominal case ($x_0\in\mathbb{X}_N$, $w\equiv 0$), Inequality~\eqref{eq:constraint_violation} ensures that the constraint violation becomes arbitrarily small as we increase the weight $\lambda$, similar to soft-constrained MPC~\cite[Cor.~2]{feller2016relaxed}. 
Furthermore, Inequality~\eqref{eq:constraint_violation} ensures that the constraint violation scales linearly with the distance of the initial state to the feasible set $\|x_0\|_{\mathbb{X}_N}^2$ and the magnitude of the disturbances $\|w(k)\|^2$.
Notably,  this bound is uniform w.r.t. the weight $\lambda$, i.e., even an arbitrarily large weight $\lambda$ does not result in degenerate behaviour. 
{This is not the case for standard soft-constrained MPC formulations; see the discussion in Section~\ref{sec:linear}.}
\begin{remark}\label{rk:robust_design}
{(Incorporating a robust design)
The bound~\eqref{eq:constraint_violation} is largely qualitative:  even if a disturbance bound is known, quantitative bounds on the constraint violation are challenging to characterize.  
It is possible to obtain quantitative bounds by incorporating standard robust MPC designs, such as~\cite{kouvaritakis2016model,bayer2013discrete,singh2023robust}, into the MPC formulation~\eqref{eq:MPC_proposed}. 
Thus, robust constraint satisfaction can be ensured for a user-specified disturbance bound, while recursive feasibility and robust stability remain valid also in case of unbounded disturbances, see~\ifbool{arxiv}{App.~\ref{app:robust}}{\cite[App.~B]{JK_Soft_extended}} for details.} 
\end{remark} 
\subsection{Performance analysis}
\label{sec:theory_2}
The following theorem derives a bound on the closed-loop performance.%
\begin{theorem}
\label{thm:performance}
Let Assumptions~\ref{ass:standing}, \ref{ass:terminal}, \ref{ass:increm_stab}, and \ref{ass:exp_bounds}\ref{ass:exp_bound_ell_quad},\ref{ass:exp_bounds_alpha_delta} hold. 
For any initial condition $x_0\in\mathbb{X}_N$, weighting $\lambda>0$, and any disturbances $w(k)\in\mathbb{R}^n$, $k\in\mathbb{I}_{\geq 0}$,  {and any time $T\in\mathbb{I}_{\geq 0}$}, the closed-loop system~\eqref{eq:closed_loop} satisfies
\begin{align}
\label{eq:performance_slack}
&\dfrac{1}{1+\epsilon_\lambda}\sum_{k=0}^{{T-1}}\ell(x(k),u(k))\leq V_{\nom}(x_0)+\lambda c_{4,\delta}\sum_{k=0}^{{T-1}} \|w(k)\|^2,\nonumber\\
&\epsilon_\lambda:= \dfrac{1}{\lambda}\frac{\sigma_{\max}(Q)}{c_{3,\delta}}\in(0,\infty).
\end{align}
\end{theorem}
\begin{pf}
With $x_0\in\mathbb{X}_N$, a feasible candidate solution for Problem~\eqref{eq:MPC_equ} is given by $\bar{x}=x_0$, which implies $V_{\slack}(x_0)\leq V_{\nom}(x_0)$. 
Abbreviate $\bar{x}(k):=\bar{x}_{\slack}^*(x(k))$ and $\Delta x(k):=x(k)-\bar{x}(k)$. 
Inequality~\eqref{eq:global_decrease_1} and $\alpha_{3,\delta},\alpha_{4,\delta}$ quadratic imply 
\begin{align*}
&V_{\slack}(x(k+1))-V_{\slack}(x(k))\\
\leq& -\ell(\bar{x}(k),u(k))-\lambda c_{3,\delta}\|\Delta x(k)\|^2+\lambda c_{4,\delta}\|w(k)\|^2.
\end{align*}
Using a telescopic sum and $V_{\slack}\geq 0$, we arrive at 
\begin{align}
\label{eq:performance_telescopic}
&\sum_{k=0}^{{T-1}} \ell(\bar{x}(k),u(k))+\lambda c_{3,\delta}\|\Delta x(k)\|^2\\
\leq &V_{\slack}(x_0)+\sum_{k=0}^{{T-1}}\lambda c_{4,\delta}\|w(k)\|^2\nonumber\\
\leq & V_{\nom}(x_0)+\sum_{k=0}^{{T-1}}\lambda c_{4,\delta}\|w(k)\|^2.\nonumber
\end{align}
Given that $\ell$ is quadratic, we can use Cauchy-Schwarz and Young’s inequality with some $\epsilon>0$ to ensure:
\begin{align}
\label{eq:cauchy}
\frac{1}{1+\epsilon}\ell(x(k),u(k))\leq \ell(\bar{x}(k),u(k))+\frac{1}{\epsilon}\|\Delta x(k)\|_Q^2.
\end{align}
Choosing  $\epsilon=\epsilon_{\lambda}:=\frac{1}{\lambda}\frac{\sigma_{\max}(Q)}{c_{3,\delta}}>0$ yields
\begin{align*}
&\dfrac{1}{1+\epsilon_\lambda}\ell(x(k),u(k))\leq \ell(\bar{x}(k),u(k))+\lambda c_{3,\delta}\|\Delta x(k)\|^2.
\end{align*}  
Using this inequality in~\eqref{eq:performance_telescopic} implies~\eqref{eq:performance_slack}. $\hfill\square$
\end{pf}
In the absence of disturbances ($w\equiv 0$), the performance bound~\eqref{eq:performance_slack} is almost equivalent to the nominal performance bound~\eqref{eq:nom_MPC_performance}. The difference lies in the factor $\frac{1}{1+\epsilon_\lambda}$, which vanishes as we increase the weight $\lambda>0$. 
As discussed in Remark~\ref{rk:local_opt}, the nominal MPC performance~\eqref{eq:nom_MPC_performance} is near-optimal. 
Hence,  by increasing the weight $\lambda>0$ we achieve near-optimal performance in the nominal setting ($x_0\in\mathbb{X}_N$, $w\equiv 0$). 
{However, a large weight $\lambda$ may also increase the effect of disturbances, see Remark~\ref{rk:weight_lambda} for guidelines.}

\section{Constrained linear quadratic regulator}
\label{sec:linear}
In the following, we study the important special case of the constrained LQR and show how the design simplifies. Furthermore, we discuss the relation to soft-constrained MPC formulations.
{For the following discussion, suppose we have a linear system $f(x,u)=Ax+Bu$, the matrix $A$ is Schur stable, we have a quadratic stage cost $\ell(x,u)=\|x\|_Q^2+\|u\|_R^2$ with $Q,R$ positive definite, and the constraints $\mathbb{U}$, $\mathbb{X}$ are polytopes that satisfy $0\in\text{int}\left(\mathbb{X}\times\mathbb{U}\right)$.}  

\textit{Design procedure:} 
We compute a quadratic terminal penalty $V_{\mathrm{f}}$ based on the LQR and choose $\mathbb{X}_{\mathrm{f}}$ as the corresponding maximal positive invariant set, where the constraints $\mathbb{U},\mathbb{X}$ are satisfied. 
Thus, the nominal MPC is locally equivalent to the (infinite-horizon optimal) constrained LQR~\cite{sznaier1987suboptimal,chmielewski1996constrained,scokaert1998constrained}. 
We compute a quadratic function $V_\delta(x,z)=\|x-z\|_{P}^2$ using the Lyapunov equation $A^\top PA+Q-P=0$. 
This design directly satisfies all the  assumptions.  
Compared to nominal MPC, this adds minimal design effort, only one Lyapunov equation to compute $V_\delta$.
The resulting optimization problem~\eqref{eq:MPC_proposed} is a standard QP. 
The closed-loop system is ISS and provides uniform bounds on the constraint violation for arbitrary disturbances $w$ and initial conditions $x_0$ (cf. Theorems~\ref{thm:global} and \ref{thm:constraints}). 
Furthermore, with a large weight $\lambda>0$, we achieve close-to-optimal performance on the feasible set $\mathbb{X}_N$ (cf. Thm.~\ref{thm:performance}).

\subsection{Comparison to soft-constrained MPC}
\label{sec:linear_soft}
In the following, we {compare against} state-of-the-art soft-constrained MPC formulations, which are also considered in the numerical comparison later. 
The state constraints $\mathbb{X}=\{x|~H x\leq h\}$ can be relaxed by introducing a slack variable $\xi=\max\{H x-h,0_p\}\in\mathbb{R}^p_{\geq 0}$ and adding a quadratic penalty $\|\xi\|_{Q_\xi}^2$ with $Q_\xi$ positive definite~\cite{scokaert1999feasibility,de1994constraint}. 

\SoftP:
Using such soft state constraints with a terminal set constraint based on the LQR (cf.~\cite{sznaier1987suboptimal,chmielewski1996constrained,scokaert1998constrained}) increases the set of feasible {initial states}, while approximately retaining the local optimality guarantees. 
We refer to this approach as \SoftP, as it approximately recovers the nominal \textbf{P}erformance. 

\SoftT:
A formulation to partially relax this terminal set constraint while keeping the local LQR performance has been proposed in~\cite{zeilinger2014soft}, which was further refined to a polytopic terminal set and nonlinear systems in~\cite[Sec.~3]{wabersich2021soft} and~\cite[Sec.~3]{chatzikiriakos2024learning}. 
This approach uses a terminal penalty based on the LQR and a hard terminal set constraint that only accounts for input constraints, while the future state constraint violations are penalized similar to the peak-based cost in~\cite{zheng1995stability}. 
We refer to this approach as \SoftT due to the softened \textbf{T}erminal constraint. 
An alternative design to relax the terminal set constraint is proposed in~\cite{rakovic2021model} by imposing linear dynamics on the slack variables. 

\SoftG: 
By combining the globally stabilizing MPC~\cite{zheng1995stability,rawlings1993stability} with the terminal penalty in~\cite[Thm.~5]{feller2016relaxed}, one can construct a globally feasible soft-constrained MPC formulation:  
Considering $\|x\|_Q^2+\|\xi\|_{Q_\xi}^2\leq \|x\|^2_{Q+H^\top Q_\xi H}$, a globally valid terminal penalty is given by $V_{\mathrm{f}}(x)=\|x\|_{P_{\mathrm{f}}}^2$ with $A^\top P_{\mathrm{f}} A-P_{\mathrm{f}}+Q+H^\top Q_\xi H=0$, $\mathbb{X}_{\mathrm{f}}=\mathbb{R}^n$, $u=k_{\mathrm{f}}(x)=0$. 
We refer to this approach as \SoftG due to its \textbf{G}lobal properties. 

\textit{Qualitative comparison:}
The offline design requirements and online computational complexity of these different formulations are similar to nominal MPC. 
The \SoftP and \SoftT designs yield closed-loop operation equivalent to a nominal MPC under nominal operating conditions, ensuring local optimality (Rk.~\ref{rk:local_opt}). 
Despite the relaxed state constraints, the presence of unexpected disturbances $w$ in combination with the required terminal state constraint $\mathbb{X}_{\mathrm{f}}$ can yield infeasibility. 
On the other hand, the \SoftG approach ensures ISS also for {unbounded} disturbances $w$, although performance deteriorates significantly {compared} to nominal MPC. 
The proposed approach simultaneously inherits the positive properties of both approaches: 
Under nominal operating conditions, we recover nominal MPC performance (cf. Thm.~\ref{thm:performance}), which is locally optimal. 
Reliable operation and robust stability (Thm.~\ref{thm:global},\ref{thm:constraints}) are ensured without making any assumption on the disturbances or the initial condition. 

An additional feature of the theoretical analysis (Thm.~\ref{thm:global}, \ref{thm:constraints}) is the fact that stability properties and bounds on the constraint violation hold uniformly, even for arbitrarily large weights $\lambda>0$. 
This is in general not the case for existing soft-constrained MPC formulations.
Specifically, for the discussed soft-constrained MPC formulations, the Lyapunov function results in a decay of the form $-\|x\|_Q^2$. 
The increase due to disturbances can be bounded proportional to $\|w\|_{Q+Q_\xi}$. As such, the stability properties of existing soft-constrained MPC formulations can deteriorate if a large penalty $Q_\xi$ is chosen, which is not the case in the proposed formulation.

\section{Numerical examples} 
\label{sec:num}
First, we consider a simple linear system from~\cite{wabersich2021soft} to provide a quantitative comparison to existing soft-constrained MPC approaches. 
Then, we showcase robustness to large disturbances using the nonlinear system from~\cite{limon2018nonlinear}, where we also demonstrate feasibility issues of soft-constrained MPC schemes. 
Implementation details are provided in\ifbool{arxiv}{ Appendix~\ref{app:num}}{\cite[App.~C]{JK_Soft_extended}}. 
The code is available online:\\{\url{https://github.com/KohlerJohannes/SoftMPC}}. 

\subsection{Comparison to linear soft-constrained MPC}
\label{sec:num_linear}
\begin{figure*}[hbtp]
\begin{center}
\includegraphics[width=0.3\textwidth]{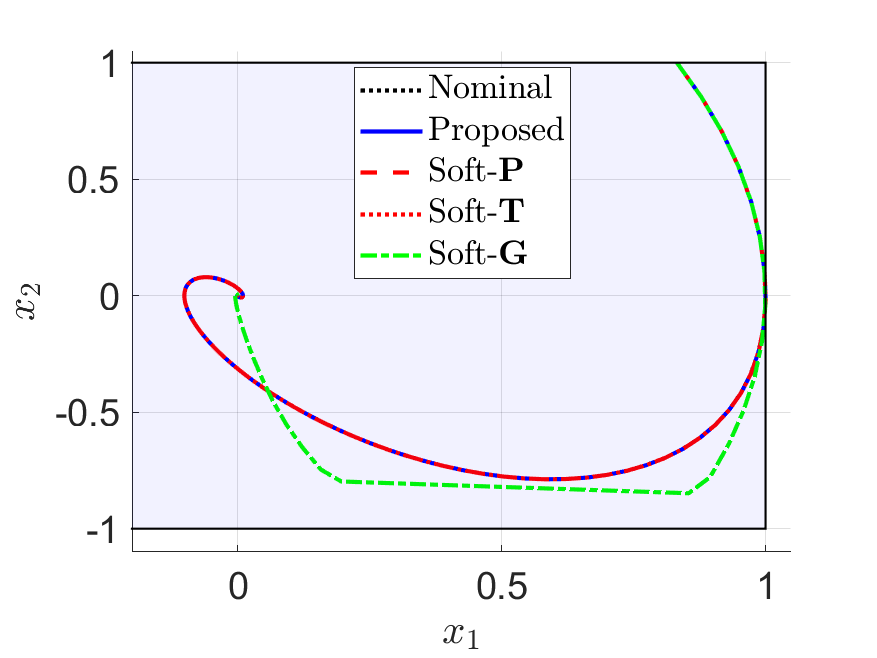}\quad 
\includegraphics[width=0.3\textwidth]{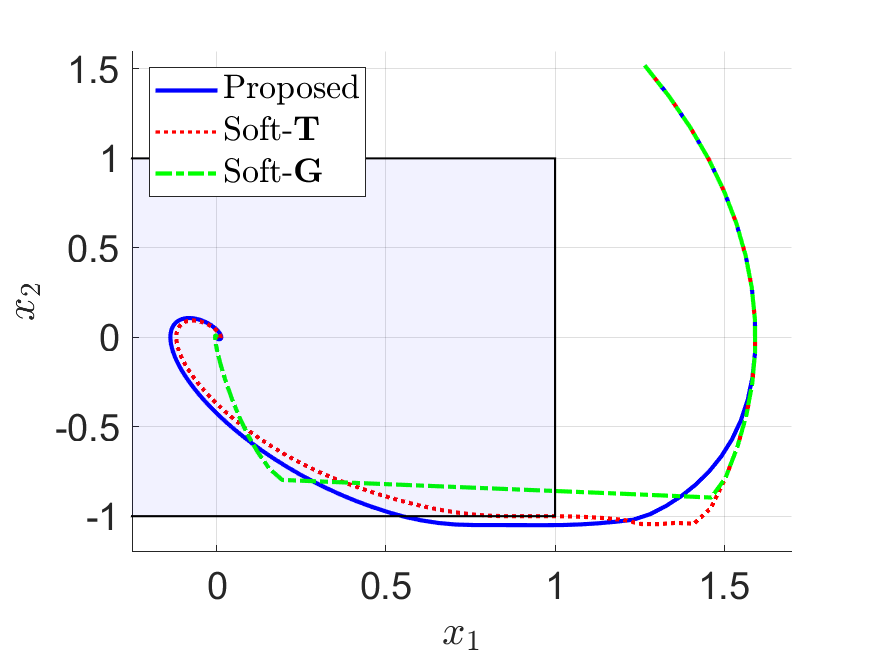}\quad
\includegraphics[width=0.3\textwidth]{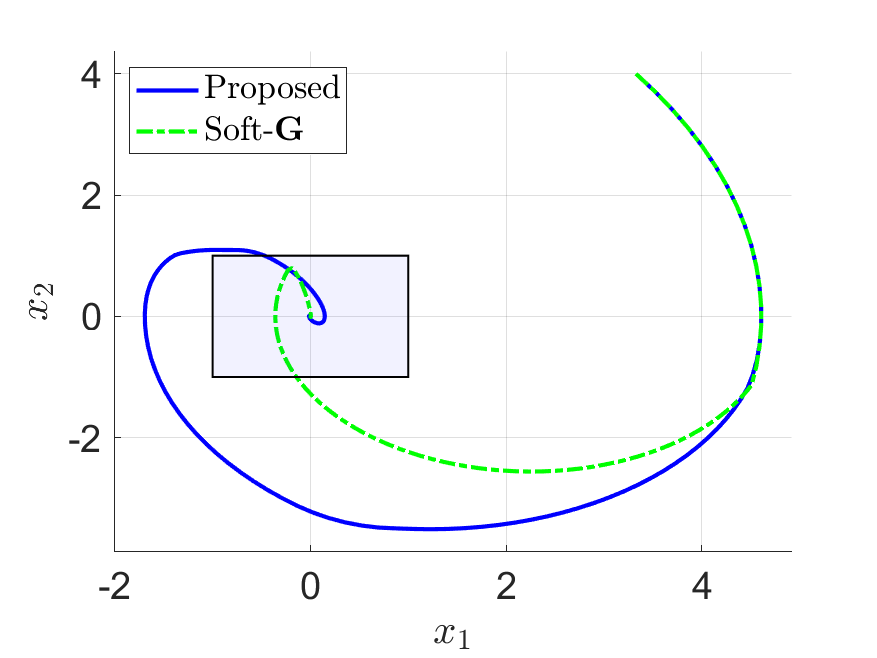}
\end{center}
\caption{{Linear system}: Closed-loop trajectories initialized at $x_0=c\cdot [0.832,1]^\top$ with $c=1$ (left), $c=1.520$ (middle), $c=4$ (right).
The state constraint $\mathbb{X}$ is shaded. Only feasible MPC solutions are plotted.}
\label{fig:soft_linear}
\end{figure*}
We compare the proposed MPC formulation to the soft-constrained MPC formulations for linear systems introduced in Section~\ref{sec:linear}:
{\begin{itemize}
\item locally optimal (\SoftP, \cite{sznaier1987suboptimal,chmielewski1996constrained,scokaert1998constrained});
\item locally optimal with relaxed terminal constraint (\SoftT, \cite{zeilinger2014soft,wabersich2021soft});  
\item  globally feasible (\SoftG, \cite{zheng1995stability,rawlings1993stability,feller2016relaxed}).
\end{itemize}}
We consider a mass-spring-damper system taken from~\cite{wabersich2021soft} with bounded control force, soft state constraints on position and velocity, and no disturbances, see \ifbool{arxiv}{Appendix~\ref{app:num}}{\cite[App.~C]{JK_Soft_extended}} for details.

We compare the performance and region of attraction using three exemplary initial conditions
$x_0=c\cdot [0.832,1]^\top$, $c\in\{1,1.520,4\}$: on the boundary of the nominal feasible set $\mathbb{X}_N$, on the boundary of the feasible set of the \SoftT approach, and a large initial condition to approximately consider the global behaviour. 
The resulting closed-loop trajectories can be seen in Figure~\ref{fig:soft_linear}. 
{A quantitative performance comparison is provided in Table~\ref{tab:linear_perf} and computational demand is reported in Table~\ref{tab:linear_comp}.}
\begin{table}[h]
\caption{{Closed-loop cost for different soft-constrained MPC approaches starting at three different initial conditions, indicating by the scaling $c$. The cost is normalized with respect to the cost of the proposed approach. 
$\times$ indicates infeasibility of the approach.}}
\label{tab:linear_perf}
\begin{tabular}{c|c|c|c|c}
&Proposed&\SoftP&\SoftT&\SoftG\\\hline
$c=1.00$&$100\%$&$100.0\%$&$100.0\%$&$116.2\%$\\
$c=1.52$&$100\%$&$\times$&$102.4\%$&$109.0\%$\\
$c=4.00$&$100\%$&$\times$&$\times$&$96.7\%$.
\end{tabular}
\end{table}
\begin{table}[ht]
\caption{{Average computation times  in [ms] based on a $101\times 101$ grid of $\mathbb{X}$, considering $9863$ feasible points $x\in\mathbb{X}_N$.}}
\label{tab:linear_comp}
\begin{tabular}{c|c|c|c|c|c}
Nominal&Proposed&\SoftP&\SoftT&\SoftG\\\hline
$3.73$&$3.98$&$7.42$&$10.79$&$6.91$
\end{tabular}
\end{table}\\
In case $c=1$, we have $x_0\in\mathbb{X}_N$. 
The proposed approach and the soft-constrained MPC formulations using an LQR based terminal penalty (\SoftP, \SoftT) are approximately equivalent to the nominal MPC on the feasible set. 
On the other hand, the globally feasible soft-constrained MPC (\SoftG) requires a very large terminal penalty (factor $10^6$ larger) resulting in a significant performance deterioration. 
We note that the \SoftP approach looses feasibility for $c>1.02$. 
Hence, we directly consider $c=1.52$, on the boundary of the feasible set of \SoftT. 
The soft-constrained MPC formulations \SoftG and \SoftT directly minimize the distance to the state constraints $\mathbb{X}$, which reduces the constraint violation by $5$--$8\%$. 
On the other hand, the proposed formulation achieves a better performance. 
Considering the large initial condition with $c=4$, only the proposed approach and the \SoftG implementation are feasible. 
In this case, the \SoftG approach achieves slightly better performance by using a bang-bang control.
The computation complexity of the proposed method is almost equivalent to a nominal MPC, while the slack variables significantly increase the complexity of the soft-constrained MPC alternatives.

\textit{Summary:}
The overall design and online optimization of the proposed approach is simpler or comparable to the considered soft-constrained MPC alternatives. 
On the nominal feasible set $\mathbb{X}_N$, the proposed approach, \SoftP and \SoftT approximately recover the nominal MPC, while the \SoftG approach resulted in a significant performance deterioration. 
For very large initial states (or equivalently outlier disturbances), only the proposed approach and \SoftG are feasible. 
The proposed approach is the only method, which achieves both: local optimality and an arbitrary large (global) region of attraction.

\subsection{Nonlinear system subject to large disturbances}
\label{sec:num_nonlinear}
The following example demonstrates the practical benefits for nonlinear systems subject to large disturbances. 
We consider the nonlinear four-tank system from the experiments in~\cite[Sec.~VI]{limon2018nonlinear}. 
The input $u\in\mathbb{R}^2_{\geq 0}$ corresponds to the water flow, which is subject to the hard constraint. 
The states $x\in\mathbb{R}^4_{\geq 0}$ correspond to the water level in different tanks and are subject to soft state constraints with a compact set $\mathbb{X}$. 
The problem addresses a coupled multi-input-multi-output nonlinear open-loop stable system with non-minimum phase behaviour. 
We consider a quadratic function $V_\delta(x,z)=\|x-z\|_{P_\delta}^2$, which satisfies Assumption~\ref{ass:increm_stab} in the considered region of operation, see \ifbool{arxiv}{Appendix~\ref{app:num}}{\cite[App.~C]{JK_Soft_extended}}.  
We consider four MPC implementations: 
\begin{itemize}
\item Nominal: a nominal MPC implementation;
\item Soft: a nominal MPC with soft state constraints;
\item Proposed: the proposed formulation (Section~\ref{sec:proposed});
\end{itemize} 
The nominal MPC scheme is designed as in~\cite[Sec.~VI]{limon2018nonlinear} using a tracking formulation to achieve a desired reference $[x_1,x_3]^\top\stackrel{!}{=}y_{\mathrm{d}}=[1.3,1.3]^\top$, which lies on the boundary of the state constraints $\mathbb{X}$. 
We consider uniformly distributed disturbances $\|w(k)\|_\infty\leq \overline{w}(k)$. 
The disturbance bound $\bar{w}(k)$ is zero for $k\in\mathbb{I}_{[0,k_1]}$, then linearly increases to $5\cdot 10^{-2}$ in $k\in\mathbb{I}_{[k_1,k_2]}$, linearly decreases to zero in $k\in\mathbb{I}_{[k_2,k_3]}$ and stays zero for $k\in\mathbb{I}_{\geq k_3}$ with $k_{i}\in\{50,350,650\}$.  

{The resulting closed-loop trajectories can be seen in Figure~\ref{fig:tank}.}  
In the interval $k\in\mathbb{I}_{[0,k_1]}$ (no disturbances), the nominal MPC successfully converges to the desired setpoint $y_{\mathrm{d}}$. 
{A quantitative comparison of computational demand and performance in this interval can be found in Table~\ref{tab:nonlinear}.} 
The soft-constrained MPC and the proposed MPC are virtually indistinguishable from the nominal MPC in this nominal setting ($x\in\mathbb{X}_N$). 
In the interval $k\in\mathbb{I}_{[k_1,k_2]}$, the magnitude of the disturbances increases and the nominal MPC quickly becomes infeasible. 
Both the proposed MPC framework and the soft-constrained MPC result in trajectories with fluctuations and constraint violations. 
Despite the relaxation of the state constraints, the soft-constrained MPC has a bounded feasible set due to the required terminal set constraint. 
Thus, the persistent disturbances $w$ result in a loss of feasibility. 
On the other hand, the proposed MPC framework retains feasibility and stability irrespectively of the magnitude of the disturbances. 
%
%
\begin{table}[ht]
\caption{{Quantitative comparison over the interval $k\in\mathbb{I}_{[0,k_1]}$, where all approaches are feasible. Average computation times in [ms]. The tracking error $\|[x_1,x_3]^\top-y_{\mathrm{d}}\|^2$ is summed over the interval $k\in\mathbb{I}_{[0,k_1]}$ and normalized with respect to the cost  of the nominal MPC.}}
\label{tab:nonlinear}
 \begin{tabular}{c|r|r|r|r}
&Nominal&Soft&Proposed\\\hline
Computation time&$13.0$&$16.3$&$15.5$\\
Tracking error&$100\%$&$99.78\%$&$99.86\%$
\end{tabular}
\end{table}\\
We have demonstrated the practical benefits of the proposed MPC framework for nonlinear systems subject to large disturbances.
During nominal operation, the proposed softened initial state constraint and standard soft state constraints approximately result in the same closed-loop trajectories as a nominal MPC scheme. 
The soft-constrained MPC scheme lost feasibility in the presence of large disturbances. 
On the other hand, the proposed relaxed initial state constraint ensures feasibility and (input-to-state) stability, even under large disturbance. 
 This makes the proposed approach particularly attractive in practice, as no infeasibility handling is required. 

\begin{figure}[t] 
\begin{center}
\includegraphics[width=0.38\textwidth]{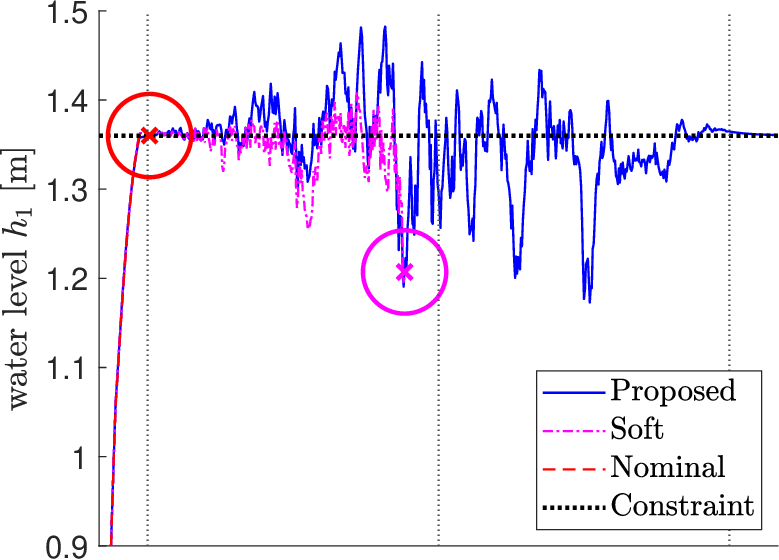}\\
\vspace{2mm}
\includegraphics[width=0.38\textwidth]{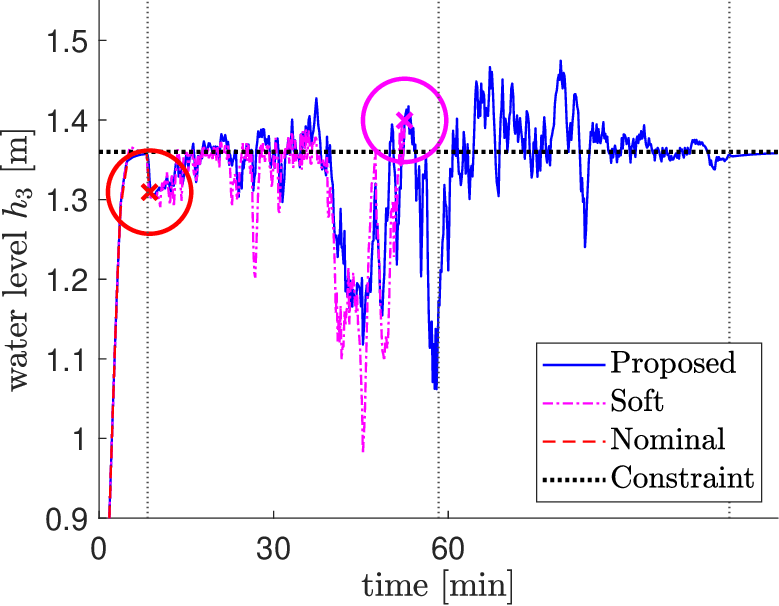}
\end{center}
\caption{{Nonlinear system subject to disturbances: }Closed-loop water level $h_1$ and $h_3$ for different MPC formulations. 
State constraints and time points $k_i$, $i\in\mathbb{I}_{[1,3]}$, are shown as black dotted lines. 
The MPC schemes are stopped in case of infeasibility, which is highlighted with a circle with a cross $\times$.}
\label{fig:tank}
\end{figure}

\section{Conclusion}  
\label{sec:sum}
We showed that relaxing the initial state constraint in MPC with a suitable penalty provides an inherently robust MPC formulation. 
Under nominal operating conditions, the proposed approach shares the (close-to-optimal) performance and constraint satisfaction of a nominal MPC design. 
{The softened initial state constraint avoids loss of feasibility, ensures ISS, and provides a suitable bound on the cumulative constraint violation, even under unbounded disturbances.} 
{The idea is presented for open-loop stable nonlinear systems using a Lyapunov function as the penalty, but the idea naturally generalizes to Lyapunov stable systems, stabilizable systems, and an implicit trajectory-based characterization (cf.~\ifbool{arxiv}{App.~\ref{app:stability}}{\cite[App.~A]{JK_Soft_extended}}).} 
Benefits compared to soft-constrained MPC formulations have been demonstrated in numerical examples. 
\ifbool{arxiv}{Open issues include extending the inherent robustness analysis to more general (economic) costs~\cite{mcallister2023suboptimal}, noisy measurements~\cite{roset2008robustness}, and parametric model mismatch~\cite{kuntz2024beyond}.}{} 
\bibliographystyle{unsrt}
\bibliography{Literature} 
\ifbool{arxiv}{\appendix
\section{Relax Assumption~\ref{ass:increm_stab}}
\label{app:stability}
The following results relax Assumption~\ref{ass:increm_stab}, which considers a known Lyapunov function $V_\delta$ certifying asymptotic incremental stability.
Appendix~\ref{app:implicit} shows how explicit knowledge of $V_\delta$ can be relaxed. 
Appendix~\ref{app:Lyap_stable} and \ref{app:input} relax the asymptotic stability condition to Lyapunov stable systems and stabilizable systems, respectively. 
\subsection{Implicit initial state penalty}
\label{app:implicit}
In the following, we show how the proposed MPC formulation~\eqref{eq:MPC_proposed} can be implemented without requiring an explicit characterization of the incremental Lyapunov function $V_\delta$. 
This facilitates the overall design of the proposed MPC formulation, especially for nonlinear systems.

We use an implicit characterization of $V_\delta(x,\bar{x})$ by predicting the open-loop trajectories starting from both initial conditions over some finite horizon $M\in\mathbb{I}_{\geq 1}$ with an input sequence $\mathbf{u}\in\mathbb{U}^M$:
\begin{align}
\label{eq:increm_implicit}
V_{\delta,\impl}(x,\bar{x},\mathbf{u}):=&\sum_{k=0}^{M-1}\|x_{\mathbf{u}}(k,x)-x_{\mathbf{u}}(k,\bar{x})\|^2.
\end{align} 
This design is inspired by the converse Lyapunov theorem in~\cite{angeli2002lyapunov} (cf. also~\cite[App.~C]{kohler2021dynamic}), and the approximate terminal cost design in~\cite[Prop. 4.34]{kohler2021dynamic}.  
\begin{assumption}
\label{ass:exp_stab}
(Incremental exponential stability)
There exist constants $C\geq 1$, $\rho\in[0,1)$, such that for any $x,\bar{x}\in\mathbb{R}^n$ and any $\mathbf{u}\in\mathbb{U}^\infty$, it holds
\begin{align}
\label{eq:increm_exp_stable}
\|x_{\mathbf{u}}(k,x)-x_{\mathbf{u}}(k,\bar{x})\|^2\leq C\rho^k\|x-\bar{x}\|^2,\quad k\in\mathbb{I}_{\geq 0}.
\end{align}
\end{assumption}
Assumption~\ref{ass:exp_stab} reflects the fact that the system is \textit{exponentially} incrementally stable~\cite{angeli2002lyapunov,tran2019convergence}. For the following derivation, we assume that a rough upper bound for $C,\rho$ is known. 
Given an input sequence $\mathbf{u}\in\mathbb{U}^L$, we use $\mathbf{u}_{[k_1,k_2]}\in\mathbb{U}^{k_2-k_1+1}$ to denote the subsequence starting at $k_1$ and ending at $k_2$ with $k_1,k_2,L\in\mathbb{I}_{\geq 1}$, $k_2\leq L-1$. 
\begin{lemma}
\label{lemma:increm_implicit}
Let Assumption~\ref{ass:exp_stab} hold and suppose $M>\log(C)/\log(\rho)$. 
There exist constants $c_{\delta,1},c_{\delta,2},c_{\delta,3},c_{\delta,4}>0$, such that for any $x,\bar{x}\in\mathbb{R}^n$, and any $\mathbf{u}\in\mathbb{U}^{M+1}$, it holds:
\begin{subequations}
\label{eq:increm_implicit_ineq}
\begin{align}
\label{eq:increm_implicit_ineq_1}
&c_{\delta,1}\|x-\bar{x}\|^2\leq V_{\delta,\impl}(x,\bar{x},\mathbf{u}_{[0:M-1]})\leq c_{\delta,2}\|x-\bar{x}\|^2,\\
\label{eq:increm_implicit_ineq_2}
&V_{\delta,\impl}(f(x,\mathbf{u}_0)+w,f(\bar{x},\mathbf{u}_0),\mathbf{u}_{[1,M]})\nonumber\\
\leq& V_{\delta,\impl}(x,\bar{x},\mathbf{u}_{[0:M-1]})-c_{\delta,3}\|x-\bar{x}\|^2+c_{\delta,4}\|w\|^2.
\end{align}
\end{subequations}
\end{lemma}
\begin{pf}
The lower bound in Inequality~\eqref{eq:increm_implicit_ineq_1} holds with $c_{\delta,1}:=1$ using $M\geq 1$. 
The upper bound holds with $c_{\delta,2}:=\frac{1-\rho^M}{1-\rho}C$ using Inequality~\eqref{eq:increm_exp_stable} with a geometric series. 
Denote $\tilde{\mathbf{u}}:=\mathbf{u}_{[1,M]}$ and note that 
\begin{align}
\label{eq:increm_implicit_extend}
&V_{\delta,\impl}(f(x,\mathbf{u}_0),f(\bar{x},\mathbf{u}_0),\tilde{\mathbf{u}})
-V_{\delta,\impl}(x,\bar{x},\mathbf{u}_{[0:M-1]})\nonumber\\ 
=&\|x_{\mathbf{u}}(M,x)-x_{\mathbf{u}}(M,\bar{x})\|^2-\|x-\bar{x}\|^2\nonumber\\
\stackrel{\eqref{eq:increm_exp_stable}}{\leq}& (C\rho^M-1)\|x-\bar{x}\|^2.
\end{align}
Furthermore, for any $\epsilon>0$ we have
\begin{align*}
&\dfrac{1}{1+\epsilon}V_{\delta,\impl}(f(x,\mathbf{u}_0)+w,f(\bar{x},\mathbf{u}_0),\tilde{\mathbf{u}})\\
=&\dfrac{1}{1+\epsilon}\sum_{k=0}^{M-1}\|x_{\tilde{\mathbf{u}}}(k,f(x,\mathbf{u}_0)+w)-x_{\tilde{\mathbf{u}}}(k,f(\bar{x},\mathbf{u}_0))\|^2\\
\leq &\sum_{k=0}^{M-1}\|x_{\tilde{\mathbf{u}}}(k,f(x,\mathbf{u}_0))-x_{\tilde{\mathbf{u}}}(k,f(\bar{x},\mathbf{u}_0))\|^2\\
&+\dfrac{1}{\epsilon}\sum_{k=0}^{M-1}\|x_{\tilde{\mathbf{u}}}(k,f(x,\mathbf{u}_0)+w)-x_{\tilde{\mathbf{u}}}(k,f(x,\mathbf{u}_0))\|^2\\
=&V_{\delta,\impl}(f(x,\mathbf{u}_0),f(\bar{x},\mathbf{u}_0),\tilde{\mathbf{u}})\\
&+\dfrac{1}{\epsilon}V_{\delta,\impl}(f(x,\mathbf{u}_0)+w,f(x,\mathbf{u}_0))\\
{\leq} &V_{\delta,\impl}(x,\bar{x},\mathbf{u}_{[0,M-1]})-(1-C\rho^M)\|x-\bar{x}\|^2+\dfrac{c_{\delta,2}}{\epsilon}\|w\|^2,
\end{align*}
where the first inequality used Cauchy-Schwarz and Young’s inequality (cf.~\eqref{eq:cauchy}) and the last inequality leveraged~\eqref{eq:increm_implicit_ineq_1},\eqref{eq:increm_implicit_extend}. 
Inequality~\eqref{eq:increm_implicit_ineq_2} follows by 
multiplying this inequality by $1+\epsilon>0$ and using the upper bound~\eqref{eq:increm_implicit_ineq_1}:
\begin{align*}
&V_{\delta,\impl}(f(x,\mathbf{u}_0)+w,f(\bar{x},\mathbf{u}_0),\mathbf{u}_{[1,M]})\\
\leq&V_{\delta,\impl}(x,\bar{x},\mathbf{u}_{[0:M-1]})+\underbrace{\dfrac{1+\epsilon}{\epsilon}c_{\delta,2}}_{=:c_{\delta,4}}\|w\|^2\\ 
&-\underbrace{\left[(1+\epsilon)(1-C\rho^M)-\epsilon c_{\delta,2}\right]}_{=:c_{\delta,3}}\|x-\bar{x}\|^2,
\end{align*}
where $c_{\delta,3}>0$ follows from $C\rho^M<1$ and choosing $\epsilon>0$ sufficiently small.
$\hfill\square$
\end{pf}
Lemma~\ref{lemma:increm_implicit} ensures that for a sufficiently large horizon $M$, $V_{\delta,\impl}$ satisfies properties comparable to an incremental ISS Lyapunov function $V_\delta$ (Asm.~\ref{ass:increm_stab})\footnote{
A standard incremental Lyapunov function is given by $V_{\delta}(x,\bar{x})=\max_{\mathbf{u}\in\mathbb{U}^M}V_{\delta,\impl}(x,\bar{x},\mathbf{u})$~\cite[App.~C]{kohler2021dynamic}.}.  
The proposed MPC formulation with an implicit characterization of $V_\delta$ is given by 
\begin{subequations}
\label{eq:MPC_proposed_implicit}
\begin{align}
&\min_{\bar{x}\in\mathbb{R}^n,\mathbf{u}\in\mathbb{U}^N}\mathcal{J}_N(\bar{x},\mathbf{u})+\lambda\sum_{k=0}^{M-1}\|x_{\mathbf{u}}(k,x)-x_{\mathbf{u}}(k,\bar{x})\|^2,\\
\text{s.t. }&x_{\mathbf{u}}(k,\bar{x})\in\mathbb{X},~ k\in\mathbb{I}_{[0,N-1]},~ 
x_{\mathbf{u}}(N,\bar{x})\in\mathbb{X}_{\mathrm{f}},\\
&\mathbf{u}_k=k_{\mathrm{f}}(x_{\mathbf{u}}(k,\bar{x})),\quad k\in\mathbb{I}_{[N,M-2]}.
\end{align}
\end{subequations}
Compared to Problem~\eqref{eq:MPC_proposed}, the penalty $V_\delta$ is replaced by the finite-horizon cost from Lemma~\ref{lemma:increm_implicit}. 
Instead of designing a Lyapunov function $V_\delta$ offline, the complexity of Problem~\eqref{eq:MPC_proposed_implicit} is increased by additionally predicting a second trajectory $x_{\mathbf{u}}(k,x)$.
To ensure that this trajectory is well-defined in case $M> N+1$, we append the known terminal control law $k_{\mathrm{f}}(x)$ (Asm.~\ref{ass:terminal}) for the remaining horizon. 
We denote the  value function and minimizers of Problem~\eqref{eq:MPC_proposed_implicit} by $V_{\impl}$, $\mathbf{u}^*_{\impl}$, $\bar{x}^*_{\impl}$. 
The closed-loop system is given by
\begin{align}
\label{eq:closed_loop_impl}
x(k+1)=f(x(k),u(k))+w(k),\nonumber\\
u(k)=\pi_{\impl}(x(k)):={\mathbf{u}}^*_{\impl,0}(x). 
\end{align} 
The following theorem shows that this MPC formulation enjoys the same closed-loop properties as Problem~\eqref{eq:MPC_proposed}. 
\begin{theorem}
\label{thm:impl}
Let Assumptions~\ref{ass:standing}, \ref{ass:terminal}, \ref{ass:exp_bounds}\ref{ass:exp_bound_MPC}\ref{ass:exp_bound_ell_quad}, and \ref{ass:exp_stab} hold. 
Furthermore, suppose $M>\log(C)/\log(\rho)$. 
Then, 
Theorems~\ref{thm:global}, \ref{thm:constraints}, \ref{thm:performance} also apply to the closed-loop system~\eqref{eq:closed_loop_impl}, ensuring
 ISS, bounded constraint violation~\eqref{eq:constraint_violation}, and the performance bound~\eqref{eq:performance_slack}.
\end{theorem} 
\begin{pf}
Abbreviate $\bar{x}=\bar{x}_{\impl}^*(x)$, $\bar{x}^+=f(\bar{x},u)$, $x^+=f(x,u)+w$, and $u=\pi_{\impl}(x)=\pi_{\nom}(\bar{x})$. 
Following the arguments in Theorem~\ref{thm:global}, we show that there exist constants $\tilde{c}_i>0$, $i\in\mathbb{I}_{[1,4]}$, such that
\begin{subequations}
\label{eq:Lyap_joint_implicit}
\begin{align}
\label{eq:Lyap_joint_implicit_1} 
\tilde{c}_{1}\|x\|^2\leq& {V}_{\impl}(x)\leq \tilde{c}_2\|x\|^2,\\
\label{eq:Lyap_joint_implicit_2}
{V}_{\impl}(x^+)\leq &V_{\impl}(x)-\tilde{c}_3\|x\|^2+\tilde{c}_4\|w\|^2,
\end{align}
\end{subequations}
which implies that $V_{\impl}$ is an ISS Lyapunov function and correspondingly the closed-loop system is ISS.
The constraint satisfaction and performance results follow then analogously to Theorems~\ref{thm:constraints} and \ref{thm:performance}, using Inequalities~\eqref{eq:Lyap_joint_implicit} and \eqref{eq:increm_implicit_ineq}. 
\\ 
\textbf{Part I: }
The lower and upper bound in Inequality~\eqref{eq:Lyap_joint_implicit_1} hold analogous to Inequality~\eqref{eq:Lyap_joint_1} in Theorem~\ref{thm:global}
using Lemmas~\ref{lemma:MPC} and \ref{lemma:increm_implicit}. \\
\textbf{Part II: }
By ${\mathbf{u}}\in\mathbb{U}^{N+1}$, we denote the minimizing input $\mathbf{u}^*(\bar{x})\in\mathbb{U}^N$ appended with terminal control law $k_{\mathrm{f}}$, i.e., $r{\mathbf{u}}_k=\mathbf{u}^*_k(\bar{x})$, $k\in\mathbb{I}_{[0,N-1]}$, ${\mathbf{u}}_N=k_{\mathrm{f}}(x_{\mathbf{u}}(N,\bar{x}))$.
Considering the feasible candidate $\tilde{\mathbf{u}}:={\mathbf{u}}_{[1,N]}\in\mathbb{U}^N$, $\bar{x}^+$, the value function satisfies
\begin{align*}
&V_{\impl}(x^+)\leq \mathcal{J}_N(\bar{x}^+,\tilde{\mathbf{u}})+\lambda V_{\delta,\impl}(x^+,\bar{x}^+,\tilde{\mathbf{u}})\\
\stackrel{\eqref{eq:nom_MPC_Lyap_2}, \eqref{eq:increm_implicit_ineq_2}}{\leq} &V_{\nom}(\bar{x})-\ell(\bar{x},\pi_{\nom}(\bar{x}))+\lambda c_{\delta,4}\|w\|^2\\
&+\lambda \left[V_{\delta,\impl}(x,\bar{x},{\mathbf{u}}_{[0,N-1]})-c_{\delta,3}\|x-\bar{x}\|^2\right]\\
\leq &V_{\impl}(x)\\
&-\underbrace{\frac{1}{2}\min\{\sigma_{\min}(Q),\lambda c_{\delta,3}\}}_{=:\tilde{c}_3}\|x\|^2+\underbrace{c_{\delta,4}\lambda}_{=:\tilde{c}_4}\|w\|^2,
\end{align*}
where the last inequality used $2\|a\|^2+2\|b\|^2\geq \|a+b\|^2 $ with $a=\bar{x}$ and $b=x-\bar{x}$. 
Note that $\tilde{c}_3>0$ since $c_{\delta,3}>0$ according to Lemma~\ref{lemma:increm_implicit}.
$\hfill\square$
\end{pf}
Compared to Problem~\eqref{eq:MPC_proposed}, the formulation in Problem~\eqref{eq:MPC_proposed_implicit} does not require an explicit formula for the incremental Lyapunov function $V_\delta$ (Asm.~\ref{ass:increm_stab}). 
Instead, it only requires the offline choice of a horizon $M>0$ (depending on the stability properties) and the online optimization requires an additional simulation to determine $x_{\mathbf{u}}(\cdot,x)$, which facilitates implementation.  
Drawbacks include the fact that \textit{exponential} stability needs to be assumed and for some systems a large horizon $M$ might be needed, which increases computational complexity. 

\subsection{Lyapunov stable systems and semi-global results}
\label{app:Lyap_stable} 
In the following, we provide a stability analysis in case the  \textit{asymptotic} stability condition (Asm.~\ref{ass:increm_stab}) is relaxed as follows. 
\begin{assumption}
\label{ass:increm_stab_marginal}
(Incremental stability)
It holds that $V_\delta(x,z)=\|x-z\|_P^2$ with $P$ positive definite.
Furthermore, for any $x,z\in\mathbb{R}^n$, and all $u\in\mathbb{U}$, it holds that
\begin{align}
\label{eq:increm_Lyap_marginal_2}
&V_{\delta}(f(x,u),f(z,u))\leq V_{\delta}(x,z).
\end{align}
\end{assumption} 
Compared to Assumption~\ref{ass:increm_stab}, Inequality~\eqref{eq:increm_Lyap_marginal_2} only requires that $V_\delta$ is non-increasing. 
This ensures applicability to many physical systems that have invariant quantities (e.g., energy), which are only Lyapunov stable. 
Although we consider a quadratic function $V_\delta$, we expect that this restriction can be relaxed using contraction metrics and the associated Riemannian energy~\cite{manchester2017control}. 
The following theorem shows (semi-global) \textit{asymptotic} stability with this relaxed condition, assuming there are no disturbances.
\begin{theorem}
\label{thm:semi_global}
Let Assumptions~\ref{ass:standing}, \ref{ass:terminal}, \ref{ass:exp_bounds}\ref{ass:exp_bound_MPC}, and \ref{ass:increm_stab_marginal} hold.
Then, for any constant $\bar{V}>0$, there exists a constant $\lambda_{\bar{V}}>0$, such that for any $\lambda>\lambda_{\bar{V}}$ and any $x_0\in\mathbb{X}_{\mathrm{ROA}}:=\{x\in\mathbb{R}^n|V_{\mathrm{slack}}(x)\leq \bar{V}\}$, the origin is exponentially stable for the closed-loop system~\eqref{eq:closed_loop} with $w(k)=0$, $k\in\mathbb{I}_{\geq 0}$. 
\end{theorem}
\begin{pf}
Note that $V_\delta$ satisfies Inequalities~\eqref{eq:increm_Lyap} with $\alpha_{1,\delta}(r)=\sigma_{\min}(P)r^2$, $\alpha_{2,\delta}=\sigma_{\max}(P)r^2$ and $\alpha_{3,\delta}=0$. 
Hence, $V_{\mathrm{slack}}$ is non-increasing using~\eqref{eq:global_decrease_1} and $\|w\|=0$, which implies that $x(k)\in\mathbb{X}_{\mathrm{ROA}}$ holds recursively.
Furthermore, Inequalities~\eqref{eq:Lyap_joint_1} remain true with $\tilde{\alpha}_1$ quadratic, i.e., $\tilde{\alpha}_1(r)=\tilde{c}_1r^2$.
In the following, we use a case distinction based on whether $x\in\mathbb{X}_N$ to derive a lower bound on $\|\bar{x}_{\slack}^*(x)\|$, which then implies exponential stability.\\
\textbf{Part I: }Consider $x\in\mathbb{X}_N$. 
Clearly, $V_{\slack}(x)\leq V_{\mathrm{norm}}(x)$ and hence 
\begin{align*}
\lambda \sigma_{\min}(P)\|x-\bar{x}^*(x)\|^2\leq &\lambda V_{\delta}(x,\bar{x}^*(x))\\
\leq& V_{\mathrm{nom}}(x)\leq c_2\|x\|^2.
\end{align*}
Using the reverse triangle inequality, we arrive at 
\begin{align*}
\|\bar{x}^*(x)\|\geq \|x\|-\|\bar{x}^*(x)-x\|\geq \left(1-\sqrt{\dfrac{c_2}{\lambda\sigma_{\min}(P)}}\right)\|x\|.
\end{align*}
\textbf{Part II: }
Consider $x\notin\mathbb{X}_N$. 
Given that $0\in\text{int}(\mathbb{X}_{\mathrm{f}})$, there exists a constant $r>0$, such that $\|\bar{x}\|\leq r$ implies $\bar{x}\in\mathbb{X}_{\mathrm{f}}\subseteq\mathbb{X}_N$. 
Consider the candidate $\bar{x}=x\dfrac{r}{\|x\|}\in\mathbb{X}_N$, which satisfies $V_{\mathrm{nom}}(\bar{x})\leq\alpha_2(r)= c_2r^2$. 
Furthermore, 
\begin{align*}
&V_{\delta}(x,\bar{x})=\|x-\bar{x}\|_P^2
=\left(1-\dfrac{r}{\|x\|}\right)^2V_\delta(x,0).
\end{align*}
Note that $x\in\mathbb{X}_{\mathrm{ROA}}$, i.e., $V_{\mathrm{slack}}(x)\leq \bar{V}$, which implies
\begin{align}
\label{eq:semigloba_interm_1} 
\tilde{c}_1\|x\|^2\stackrel{\eqref{eq:Lyap_joint_1}}{\leq}V_{\mathrm{slack}}(x)\leq \bar{V}.
\end{align}
Hence, we get 
\begin{align*}
\left(1-\dfrac{r}{\|x\|}\right)\stackrel{\eqref{eq:semigloba_interm_1}}{\leq} 1-\dfrac{r\sqrt{\tilde{c}_1}}{\sqrt{\bar{V}}}=:\rho_{\bar{V}}\in(0,1). 
\end{align*} 
Using this feasible candidate solution, we have
\begin{align}
\label{eq:semigloba_interm_2}
&\lambda V_{\delta}(x,\bar{x}_{\slack}^*(x))\leq V_{\mathrm{slack}}(x)\nonumber\\
\leq&\lambda V_{\delta}(x,\bar{x})+V_{\mathrm{nom}}(\bar{x})\leq \lambda \rho_{\bar{V}}^2V_\delta(x,0)+c_2r^2. 
\end{align}
Applying Cauchy-Schwarz and Young’s inequality with some $\epsilon>0$ to the quadratic function $V_\delta$ yields
\begin{align*}
\|x\|_P^2\leq (1+\epsilon)\|x-\bar{x}_{\slack}^*(x)\|_P^2+\dfrac{1+\epsilon}{\epsilon}\|\bar{x}_{\slack}^*(x)\|_P^2.
\end{align*}
This implies
\begin{align*}
&\dfrac{1}{1+\epsilon}\|x\|_P^2-\dfrac{1}{\epsilon}\|\bar{x}_{\slack}^*(x)\|_P^2\leq \|x-\bar{x}_{\slack}^*(x)\|_P^2\\
\stackrel{\eqref{eq:semigloba_interm_2}}{\leq}& \rho_{\bar{V}}^2\|x\|_P^2+\dfrac{c_2}{\lambda}r^2.
\end{align*}
Thus, we get
\begin{align*} 
\|\bar{x}_{\slack}^*(x)\|_P^2\geq 
\epsilon\left(\dfrac{1}{1+\epsilon}-\rho_{\bar{V}}^2\right)\|x\|_P^2-\epsilon\dfrac{c_2r^2}{\lambda}.
\end{align*}
Choosing $\epsilon=\dfrac{1-\rho_{\bar{V}}^2}{1+\rho_{\bar{V}}^2}>0$, we have $\dfrac{1}{1+\epsilon}-\rho_{\bar{V}}^2=\dfrac{1-\rho_{\bar{V}}^2}{2}>0$. 
Using additionally $\|x\|\geq r$, we get
\begin{align*}
&\sigma_{\max}(P)\|\bar{x}_{\slack}^*(x)\|^2\geq 
\epsilon\dfrac{1-\rho_{\bar{V}}^2}{2}\|x\|_P^2-\epsilon\dfrac{c_2r^2}{\lambda}\\
\geq&\epsilon\left(\dfrac{1-\rho_{\bar{V}}^2}{2}\sigma_{\min}(P)-\dfrac{c_2}{\lambda}\right)\|x\|^2.
\end{align*}
\textbf{Part III: }
Combing the results in Part I and II, we get 
\begin{align*}
\|\bar{x}_{\slack}^*(x)\|^2\geq \epsilon_{\bar{V}}\|x\|^2\quad \forall x\in\mathbb{X}_{\mathrm{ROA}},
\end{align*}
with $\epsilon_{\bar{V}}>0$ for $\lambda>\lambda_{\bar{V}}:=\dfrac{2c_2}{(1-\rho_{\bar{V}}^2)\sigma_{\min}(P)}$. 
Using Inequality~\eqref{eq:global_decrease_1} with $\alpha_{3,\delta}=0$, $w=0$, and $\alpha_\ell$ quadratic, we arrive at
\begin{align*}
V_{\mathrm{slack}}(x^+)-V_{\mathrm{slack}}(x)\leq &- c_\ell\|\bar{x}_{\slack}^*(x)\|^2\\
\leq& -c_\ell\epsilon_{\bar{V}}\|x\|^2.
\end{align*}
Combining this bound with Inequalities~\eqref{eq:Lyap_joint_1}, we obtain exponential stability with the Lyapunov function $V_{\mathrm{slack}}$.
$\hfill\square$
\end{pf}
This result ensures that for any compact set $\mathbb{X}_{\mathrm{RoA}}$, we can choose a large enough weight $\lambda$, such that the closed-loop system is exponentially stable. 
The result is \textit{semi-global} in the sense that the value required for the weight $\lambda$ and the convergence rate deteriorate as $\bar{V}\rightarrow\infty$, which is in contrast to the \textit{uniform} guarantees in Theorem~\ref{thm:global}. 

\textit{Comparison to existing works:} 
In the linear setting (Sec.~\ref{sec:linear}), Assumption~\ref{ass:increm_stab_marginal} can be ensured by choosing a positive definite matrix $P$ according to  $A^\top P A-P\preceq 0$, which is feasible if $A$ is Lyapunov stable. 
Existing stability results for Lyapunov stable systems~\cite[Thm.~4]{rawlings1993stability}, \cite[Thm.~3]{zheng1995stability} require a sufficiently large prediction horizon $N$ depending on the initial condition~\cite[Rk.~3]{zheng1995stability}. 
In~\cite[Cor.~4]{grimm2005model}, this requirement was relaxed to a uniform bound on the prediction horizon $N$ (independent of the initial condition), by utilizing an MPC formulation without terminal constraints. 
Another solution to this problem involves an adjustment of the input regularization to ensure that the corresponding updated terminal set can be reached within the prediction horizon~\cite{casavola1999global}. 
In contrast to these classical MPC results, the proposed formulation directly achieves close-to-optimal performance and semi-global asymptotic stability without requiring specific bounds on the prediction horizon $N$, restriction to linear systems, or other adjustments.

%

\subsection{Unstable systems and soft input constraints} 
\label{app:input}
In the following, we consider possible unstable systems that can be stabilized with a nonlinear feedback $\kappa$ and we allow for small/temporary violations of the input constraints $u\in\mathbb{U}$, similar to soft constrained MPC formulations~\cite{feller2016relaxed,oancea2023relaxed}. 
\begin{assumption}
\label{ass:increm_stabilizable}
(Incremental stabilizability)
There exists a control law $\kappa:\mathbb{R}^n\times\mathbb{R}^n\times\mathbb{R}^m \rightarrow\mathbb{R}^m$  and functions $\alpha_{\delta,1},\alpha_{\delta,2},\alpha_{\delta,3},\alpha_{\delta,4}\in\mathcal{K}_\infty$, such that for all 
$x,z,w\in\mathbb{R}^n$, and all $u\in\mathbb{U}$, we have
\begin{subequations}
\label{eq:increm_Lyap_stabilizable}
\begin{align}
\label{eq:increm_Lyap_stabilizable_1}
&\alpha_{\delta,1}(\|x-z\|)\leq V_{\delta}(x,z)\leq \alpha_{\delta,2}(\|x-z\|),\\
\label{eq:increm_Lyap_stabilizable_2}
&V_{\delta}(f(x,\kappa(x,z,v))+w,f(z,u))-V_{\delta}(x,z)\nonumber\\
\leq &-\alpha_{\delta,3}(\|x-z\|)+\alpha_{\delta,4}(\|w\|).\
\end{align}
\end{subequations}
\end{assumption} 
Compared to Assumption~\ref{ass:increm_stab}, condition~\eqref{eq:increm_Lyap_stabilizable_2} holds when applying a feedback $\kappa$ instead of an open-loop input. 
In addition, we  require Lipschitz continuity of $\kappa$.  
\begin{assumption} (Lipschitz continuous feedback) 
\label{ass:exp_bounds_v2} 
There exists a constant $c_\kappa\geq 0$, such that $\|\kappa(x,z,v)-v\|\leq c_\kappa\|x-z\|$ for any $x\in\mathbb{R}^n,z\in\mathbb{X},v\in\mathbb{U}$ 
\end{assumption}  
The proposed approach is still based on Problem~\eqref{eq:MPC_proposed}, however, the closed-loop applied input is based on the feedback law $\kappa$. 
The closed-loop system is given by 
\begin{align}
\label{eq:closed_loop_kappa}
x(k+1)=&f(x(k),u(k))+w(k),\\
u(k)=&\pi_{\Input}(x(k)):=\kappa(x,\bar{x}_{\slack}^*(x),\mathbf{u}^*_{\slack,0}(x)). \nonumber
\end{align}
The following result extends Theorems~\ref{thm:global}, \ref{thm:constraints}, and \ref{thm:performance} to account for the feedback $\kappa$.
\begin{theorem}
Let Assumptions~\ref{ass:standing}, \ref{ass:terminal}, and \ref{ass:increm_stabilizable} hold. 
Then, the closed-loop system~\eqref{eq:closed_loop_kappa} is ISS with the continuous ISS Lyapunov function $V_{\slack}$. 
Suppose further that Assumptions~\ref{ass:exp_bounds} and \ref{ass:exp_bounds_v2} hold. 
There exist constants $c_{\mathbb{X,U},i}\geq 0$, $i\in\mathbb{I}_{[1,3]}$, such that for any initial conditions $x_0\in\mathbb{R}^n$, weighting $\lambda>0$, and disturbances $w(k)\in\mathbb{R}^n$, $k\in\mathbb{I}_{\geq 0}$, \change{and any time $T\in\mathbb{I}_{\geq 0}$}, the closed-loop system~\eqref{eq:closed_loop_kappa} satisfies 
\begin{align}
\label{eq:constraint_violation_input}
&\sum_{k=0}^{T-1}\|x(k)\|_{\mathbb{X}}^2+\|u(k)\|_{\mathbb{U}}^2\\
\leq& c_{\mathbb{X,U},1}\|x_0\|_{\mathbb{X}_N}^2+\dfrac{c_{\mathbb{X,U},2}}{\lambda}+c_{\mathbb{X,U},3}\sum_{k=0}^{T-1} \|w(k)\|^2.\nonumber
\end{align}
Furthermore, for any $x_0\in\mathbb{X}_N$, the closed-loop system~\eqref{eq:closed_loop_kappa} satisfies
\begin{align}
\label{eq:performance_slack_input}
&\dfrac{1}{1+\epsilon_\lambda}\sum_{k=0}^{T-1}\ell(x(k),u(k))\leq V_{\nom}(x_0)+\lambda c_{4,\delta}\sum_{k=0}^{T-1} \|w(k)\|^2,\nonumber\\
&\epsilon_\lambda:= \dfrac{1}{\lambda}\frac{\sigma_{\max}(Q)+c_\kappa^2\sigma_{\max}(R)}{c_{3,\delta}}\in(0,\infty).
\end{align}
\end{theorem}
\begin{pf}
The proof follows the arguments from Lemma~\ref{lemma:equiv_relaxed} and Theorems~\ref{thm:global}, \ref{thm:constraints}, \ref{thm:performance}. \\
\textbf{Part I: }
Abbreviate $\bar{x}=\bar{x}^*_{\slack}$ and $\bar{u}=\pi_{\nom}(\bar{x})\in\mathbb{U}$. 
Analogous to Lemma~\ref{lemma:equiv_relaxed}, it holds that
$\pi_{\Input}(x)=\kappa(x,\bar{x},\bar{u})$. \\ 
\textbf{Part II: }
Inequality~\eqref{eq:Lyap_joint_1} from Theorem~\ref{thm:global} remains valid and analogous to~\eqref{eq:Lyap_joint_2} Inequality~\eqref{eq:increm_Lyap_stabilizable_2} implies
\begin{align}
\label{eq:Lyap_joint_stab_2}
&V_{\slack}(f(x,\pi_{\Input}(x))+w)-V_{\slack}(x)\nonumber\\
\leq &-\tilde{\alpha}_3(\|x\|)+\tilde{\alpha}_4(\|w\|).
\end{align}
Hence, the closed-loop system~\eqref{eq:closed_loop_kappa} is ISS with the continuous ISS Lyapunov function $V_{\slack}$. \\
\textbf{Part III: } 
Inequality~\eqref{eq:Lyap_exp_ISS} holds analogous to Theorem~\ref{thm:constraints} using~\eqref{eq:Lyap_joint_stab_2}. 
It holds that 
\begin{align*}
\|u\|_{\mathbb{U}}^2\leq \|u-\bar{u}\|^2=\|\kappa(x,\bar{x},\bar{u})-\bar{u}\|^2\stackrel{\mathrm{Asm.}~\ref{ass:exp_bounds_v2}}{\leq} c_\kappa^2\|x-\bar{x}\|^2.
\end{align*}
In combination with~\eqref{eq:constraint_violation_1}, this implies
\begin{align*}
\|x\|_{\mathbb{X}}^2+\|u\|_{\mathbb{U}}^2\stackrel{\eqref{eq:constraint_violation_1}}{\leq} (1+c_\kappa^2)\|x-\bar{x}\|^2\stackrel{\eqref{eq:constraint_violation_1}}{\leq}  \dfrac{ (1+c_\kappa^2)}{c_{\delta,1}\lambda}V_{\slack}(x).
\end{align*}
The remainder of the proof is analogous to Theorem~\ref{thm:constraints}, resulting in~\eqref{eq:constraint_violation_input} with $c_{\mathbb{X,U},i}:= (1+c_\kappa^2)c_{\mathbb{X},i}$, $i\in\mathbb{I}_{[1,3]}$. \\
\textbf{Part IV: } 
Abbreviate $\bar{x}(k)=\bar{x}^*(x(k))$, $\bar{u}(k)=\bar{\mathbf{u}}_0^*(x(k))$, $\Delta x(k):=x(k)-\overline{x}(k)$, $\Delta u(k):=u(k)-\overline{u}(k)$.
Analogous to Inequality~\eqref{eq:performance_telescopic} in Theorem~\ref{thm:performance}, the closed-loop system~\eqref{eq:closed_loop_kappa} satisfies 
\begin{align}
\label{eq:performance_telescopic_v2}
&\sum_{k=0}^{T-1} \ell(\overline{x}(k),\bar{u}(k))+\lambda c_{3,\delta}\|\Delta x(k)\|^2\\
\leq& V_{\nom}(x_0)+\sum_{k=0}^{T-1}\lambda c_{4,\delta}\|w(k)\|^2.\nonumber
\end{align}
Given that $\ell$ is quadratic, we can use Cauchy-Schwarz and Young’s inequality with some $\epsilon>0$ to ensure:
\begin{align*}
&~\dfrac{1}{1+\epsilon}\ell(x(k),u(k))\\
\leq&~\ell(\overline{x}(k),u(k))+\dfrac{1}{\epsilon}\ell(\Delta x(k),\Delta u(k))\nonumber\\
\stackrel{\mathrm{Asm.}~\ref{ass:exp_bounds_v2}}{\leq}&~\dfrac{\sigma_{\max}(Q)+c_\kappa^2\sigma_{\max}(R)}{\epsilon}\|\Delta x(k)\|^2\\
=&~c_{\delta,3}\lambda\|\Delta x(k)\|^2,\nonumber
\end{align*}
where the last equation uses $\epsilon=\epsilon_{\lambda}:=\frac{\sigma_{\max}(Q)+c_\kappa^2\sigma_{\max}(R)}{\lambda c_{3,\delta}}>0$. 
Plugging this bound into~\eqref{eq:performance_telescopic_v2} yields the performance bound~\eqref{eq:performance_slack}. $\hfill\square$
\end{pf}
We showed that the proposed framework naturally extends to unstable systems. 
Our main requirement is that we know a feedback $\kappa$ that ensures stability and a corresponding Lyapunov function $V_\delta$, which can also be relaxed considering Appendix~\ref{app:implicit}. 
We obtain the same desired guarantees: ISS and suitable bounds on the constraint violation for arbitrary disturbances $w$; and 
close-to-optimal performance on the original feasible set. 
Contrary to the formulation originally proposed  in Section~\ref{sec:proposed}, in this section we allowed for temporary violations of the input constraints $\mathbb{U}$, similar to~\cite{feller2016relaxed,oancea2023relaxed}.

\section{Incorporating a robust design}
\label{app:robust}
The theoretical results (Sec.~\ref{sec:theory}) ensure that the closed-loop properties are not fragile w.r.t. disturbances $w$. 
However, the corresponding guarantees are largely qualitative: even if a disturbance bound is known, quantitative bounds on the constraint violation are challenging to characterize. 
In the following, we address this issue by merging the proposed design~\eqref{eq:MPC_proposed} with a standard robust MPC design~\cite{bayer2013discrete,singh2023robust}.  
This combines the two complementary advantages 
\begin{itemize}
\item robust constraint satisfaction for a user specified disturbance bound;
\item recursive feasibility and robust stability also in case of \change{unbounded} disturbances.
\end{itemize}
We first present the proposed methodology and theoretical results (App.~\ref{app:robust_1}) before elaborating on the underlying robust MPC design (App.~\ref{app:robust_prelim}), providing the complete theoretical proof (App.~\ref{app:robust_proof}), and finally discussing the approach (App.~\ref{sec:robust_discussion}). 
\subsection{Robustified approach}
\label{app:robust_1}
Let us consider a disturbance bound $\mathbb{W}=\{w\in\mathbb{R}^n|~\|w\|\leq\bar{w}\}$ with some $\bar{w}>0$ and the quadratic bounds on $V_\delta$ (Asm.~\ref{ass:exp_bounds}\ref{ass:exp_bounds_alpha_delta}). 
Given this, existing robust MPC formulations~\cite{bayer2013discrete,singh2023robust} provide robust closed-loop guarantees for all disturbances satisfying $w(k)\in\mathbb{W}$ (cf. App.~\ref{app:robust_prelim}). 
The following formulation merges such a robust MPC design with the relaxed initial state constraint~\eqref{eq:MPC_proposed} to deal with outlier disturbances and provide a large region of attraction: 
\begin{align}
\label{eq:MPC_tube_relax}
&\min_{\bar{x}\in\mathbb{R}^n,{\mathbf{u}}\in\mathbb{U}^N,s\in\mathbb{R}_{\geq 0}}\mathcal{J}_N(\bar{x},{\mathbf{u}})+\lambda s^2\\ 
\text{s.t. }&x_{{\mathbf{u}}}(k,\bar{x})\in\bar{\mathbb{X}},~ k\in\mathbb{I}_{[0,N-1]},~ x_{{\mathbf{u}}}(N,\bar{x})\in\bar{\mathbb{X}}_{\mathrm{f}},\nonumber\\
&\sqrt{V_{\delta}(x,\bar{x})}\leq \bar{\delta}+s.\nonumber
\end{align}
Compared to the robust MPC formulations~\cite{bayer2013discrete,singh2023robust}, the initial state constraint is relaxed with a slack variable $s\geq 0$, which is penalized in the cost. 
Hence, Problem~\eqref{eq:MPC_tube_relax} provides a simple-to-implement modification to robust MPC schemes, which avoids feasibility issues in case of \change{unexpectedly large} disturbances. 
Compared to the formulation proposed in Section~\ref{sec:proposed}, the state constraint $\mathbb{X}$ is replaced by a tightened constraint set $\bar{\mathbb{X}}$ and the initial state constraint is relaxed by a constant $\bar{\delta}>0$. 
Specifically,  $\bar{\delta}=\frac{\sqrt{c_{\delta,4}}}{1-\rho_\delta}\bar{w}>0$ is chosen such that $\Omega=\{(x,z)|~V_\delta(x,z)\leq \bar{\delta}^2\}$, is a robustly positively invariant (RPI) set. 
Ellipsoidal inner and outer approximation of this set are given by $\underline{\Omega}_{\bar{\delta}}:=\{x\in\mathbb{R}^n|~c_{2,\delta}\|x\|^2\leq \bar{\delta}^2\}$ and $\bar{\Omega}_{\bar{\delta}}:=\{x\in\mathbb{R}^n|~c_{1,\delta}\|x\|^2\leq \bar{\delta}^2\}$ 
and the tightened constraint is formulated as $\bar{\mathbb{X}}:=\mathbb{X}\ominus\bar{\Omega}_{\bar{\delta}}$, where $\ominus$ denotes the Pontryagin difference. 
We denote an optimal nominal state and input sequence by $\bar{x}^*_{\tubeSlack}(x)$, ${\mathbf{u}}^*_{\tubeSlack}(x)$ and the value function of Problem~\eqref{eq:MPC_tube_relax} by $V_{\tubeSlack}(x)$. 
The closed-loop system is given by
\begin{align}
\label{eq:closed_loop_robust_slack}
x(k+1)=&f(x(k),u(k))+w(k),\nonumber\\
u(k)=&\pi_{\tubeSlack}(x(k)):= {\mathbf{u}}^*_{\tubeSlack,0}(x). 
\end{align}
The set of feasible nominal state $\bar{x}$ is given by $\bar{\mathbb{X}}_N\subseteq\bar{\mathbb{X}}$. 
The conditions on the terminal set (Asm.~\ref{ass:terminal_robust}) are adjusted to also satisfy the tightend constraints.
\begin{assumption}
\label{ass:terminal_robust} (Terminal ingredients)
There exists a terminal control law $k_{\mathrm{f}}:\bar{\mathbb{X}}_{\mathrm{f}}\rightarrow\mathbb{U}$ and a function $\alpha_{\mathrm{f}}\in\mathcal{K}_\infty$, such that for any $x\in\bar{\mathbb{X}}_{\mathrm{f}}$, we have
\begin{itemize}
\item Terminal penalty: $V_{\mathrm{f}}(x^+)\leq V_{\mathrm{f}}(x)-\ell(x,u)$;
\item Constraint satisfaction: $(x,u)\in\bar{\mathbb{X}}\times\mathbb{U}$;
\item Positive invariance: $x^+\in\bar{\mathbb{X}}_{\mathrm{f}}$;
\item Weak controllability: $V_{\mathrm{f}}(x)\leq \alpha_{\mathrm{f}}(\|x\|)$, $0\in\mathrm{int}(\bar{\mathbb{X}}_{\mathrm{f}})$, $\bar{\mathbb{X}}_N$ is compact;
 \end{itemize} \vspace{-2mm}
with $x^+=f(x,u)$ and $u=k_{\mathrm{f}}(x)$. 
\end{assumption} 
\begin{theorem}
\label{thm:RMPC_relax}
Let Assumptions~\ref{ass:standing}, \ref{ass:increm_stab}, \ref{ass:exp_bounds}\ref{ass:exp_bound_MPC},\ref{ass:exp_bounds_alpha_delta}, and \ref{ass:terminal_robust} hold. 
There exist constants $\tilde{\rho}\in[0,1)$, $c_{\mathrm{w}}$, $\bar{c}_{\mathbb{X},1}$, $\bar{c}_{\mathbb{X},2}$, $\bar{c}_{\mathbb{X},3}\geq 0$, such that for any initial conditions $x_0\in\mathbb{R}^n$, weighting $\lambda>0$, and any disturbances $w(k)\in\mathbb{R}^n$, $k\in\mathbb{I}_{\geq 0}$,  \change{and any time $T\in\mathbb{I}_{\geq 0}$}, the closed-loop system~\eqref{eq:closed_loop_robust_slack} satisfies: 
\begin{align}
\label{eq:Lyap_robust}
&V_{\tubeSlack}(x(k+1))
\leq\tilde{\rho} V_{\tubeSlack}(x(k))+\lambda c_{\mathrm{w}}\|w(k)\|_{\mathbb{W}}^2,\\ 
\label{eq:constraint_robust}
&\sum_{k=0}^{T-1}\|x(k)\|_{\mathbb{X}}^2
\leq \bar{c}_{\mathbb{X},1}\|x_0\|_{\bar{\mathbb{X}}_N\oplus\underline{\Omega}_{\bar{\delta}}}^2+\dfrac{\bar{c}_{\mathbb{X},2}}{\lambda}+ \bar{c}_{\mathbb{X},3}\sum_{k=0}^{T-1} \|w(k)\|_{\mathbb{W}}^2. 
\end{align}
\end{theorem}(App.~\ref{app:robust_proof})
The proof is a natural combination of Theorems~\ref{thm:global}, \ref{thm:constraints} and robust MPC results~\cite{bayer2013discrete,singh2023robust}, see Appendix~\ref{app:robust_proof} for details. 
\change{Standard robust MPC designs suppose that the initial condition is feasible and the disturbances lie in a known set, i.e., $\|w(k)\|_{\mathbb{W}}=0$, $\||x_0\|_{\bar{\mathbb{X}}_N\oplus\underline{\Omega}_{\bar{\delta}}}=0$.} 
\change{Since Inequality~\eqref{eq:constraint_robust} holds with uniform constants $\bar{c}_{\mathbb{X},i}\ge 0$, the constraint violation can be rendered arbitrarily small by increasing the weight $\lambda$.} 
Furthermore, Inequality~\eqref{eq:Lyap_robust} ensures that the value function $V_{\tubeSlack}$ exponentially converges to zero. 
\change{Thus, the state $x(k)$ converges to the known RPI set $\{x|~V_\delta(x,0)\leq \bar{\delta}^2\}\subseteq \bar{\Omega}_{\bar{\delta}}$.} 
Thus, we recover the closed-loop properties of robust MPC schemes in case $w(k)\in\mathbb{W}$.
Given a corresponding robust MPC formulation, implementation requires only softening the initial state constraint with a penalty.
This makes the proposed approach very attractive for practical application. 

\subsection{Preliminaries: Robust MPC design}
\label{app:robust_prelim}
To provide context for the proposed design and its theoretical guarantees, we recap a basic robust MPC design and its theoretical properties. 
Analogous to the rigid tube MPC formulations in~\cite{bayer2013discrete,singh2023robust}, we construct a robustly positively invariant (RPI) set using the incremental Lyapunov function (Asm.~\ref{ass:increm_stab}).
\begin{lemma}
\label{lemma:RPI}
Let Assumptions~\ref{ass:increm_stab} and \ref{ass:exp_bounds}\ref{ass:exp_bounds_alpha_delta} hold. 
There exist constants $\bar{\delta}>0$, $\rho_\delta\in[0,1)$, such that for any $x,z,w\in\mathbb{R}^n$, $u\in\mathbb{U}$, we have 
\begin{align}
\label{eq:RPI}
&\sqrt{V_\delta(f(x,u)+w,f(z,u))}- \bar{\delta}\\
\leq&\rho_\delta\max \{\sqrt{V_\delta(x,z)}-\delta,0\}+\sqrt{c_{4,\delta}}\|w\|_{\mathbb{W}}.\nonumber 
\end{align}
\end{lemma}
\begin{pf} 
Define $\rho_\delta^2=1-\frac{c_{\delta,3}}{c_{\delta,2}}\in[0,1)$, which yields 
\begin{align*}
&\sqrt{V_\delta(f(x,u)+w,f(z,u))}\stackrel{\eqref{eq:increm_Lyap_2}}{\leq} \sqrt{\rho_\delta^2 V_\delta(x,z)+c_{\delta,4}\|w\|^2}\\
\leq &\rho_\delta\sqrt{ V_\delta(x,z)}+\sqrt{c_{\delta,4}}\|w\|\\
\leq & \rho_\delta\bar{\delta}+\rho_\delta \max\{\sqrt{V_\delta(x,z)}-\bar{\delta},0\}\\
&+\sqrt{c_{\delta,4}}\bar{w}+\sqrt{c_{4,\delta}}\max\{\|w\|-\bar{w},0\}\\
= & \bar{\delta}+\rho_\delta \max\{\sqrt{V_\delta(x,z)}-\bar{\delta},0\}+\sqrt{c_{4,\delta}}\|w\|_{\mathbb{W}}, 
\end{align*}
where the last equation uses $\bar{\delta}=\frac{\sqrt{c_{\delta,4}}}{1-\rho_\delta}\bar{w}>0$. $\hfill\square$
\end{pf}
Inequality~\eqref{eq:RPI} ensures that the sublevel set defined by $\bar{\delta}$ is RPI and exponentially contracting, if the disturbances satisfy $w\in\mathbb{W}$. 
%
The robust MPC is given by
\begin{align}
\label{eq:MPC_tube}
V_{\tube}(x)=&\min_{\mathbf{u}\in\mathbb{U}^N,\bar{x}\in\mathbb{R}^n}\mathcal{J}_N(\bar{x},{\mathbf{u}}),\\
\text{s.t. }&x_{{\mathbf{u}}}(k,\bar{x})\in\bar{\mathbb{X}},~ k\in\mathbb{I}_{[0,N-1]},\nonumber\\
&V_\delta(x,\bar{x})\leq \bar{\delta}^2,~x_{{\mathbf{u}}}(N,\bar{x})\in\bar{\mathbb{X}}_{\mathrm{f}}.\nonumber
\end{align}
Compared to a nominal MPC~\eqref{eq:MPC_nom}, the robust MPC formulation is subject to tightened constraints and optimizes over the initial nominal state $\bar{x}$ within the RPI set (Lemma~\ref{lemma:RPI}). 
We denote a minimizer by $\bar{x}_\tube^*(x)\in\bar{\mathbb{X}}$, $\mathbf{u}^*_{\tube}(x)\in\mathbb{U}^N$. 
The set of feasible nominal states $\bar{x}$ is defined as $\bar{\mathbb{X}}_N\subseteq\bar{\mathbb{X}}$. 
The set of feasible states $x$ is denoted by $\bar{\mathbb{X}}_{N,\bar{\delta}}$, which satisfies $\bar{\mathbb{X}}_N\oplus\underline{\Omega}_{\bar{\delta}}\subseteq\bar{\mathbb{X}}_{N,\bar{\delta}}\subseteq\mathbb{X}$. 
The closed-loop system is given by
\begin{align}
\label{eq:closed_loop_robust}
x(k+1)=&f(x(k),u(k))+w(k),\nonumber\\
u(k)=&\pi_{\tube}(x(k)):=\mathbf{u}^*_{\tube,0}(x).
\end{align}
The following theorem establishes the closed-loop properties of the robust MPC scheme. 
\begin{theorem}
\label{thm:robust_MPC}
Let Assumptions~\ref{ass:standing}, \ref{ass:increm_stab}, \ref{ass:exp_bounds}\ref{ass:exp_bounds_alpha_delta}, and \ref{ass:terminal_robust} hold. 
Suppose further that $w(k)\in\mathbb{W}$, $k\in\mathbb{I}_{\geq 0}$ and that Problem~\eqref{eq:MPC_tube} is feasible with $x=x_0$. 
Then, Problem~\eqref{eq:MPC_tube} is recursive feasible for all $k\in\mathbb{I}_{\geq 0}$ and the state constraints are robustly satisfied, i.e., $x(k)\in\mathbb{X}$, $k\in\mathbb{I}_{\geq 0}$, for the closed-loop system~\eqref{eq:closed_loop_robust}. 
Furthermore, there exist $\bar{\alpha}_1,\bar{\alpha}_2\in\mathcal{K}_\infty$, such that for all $x\in\bar{\mathbb{X}}_{N,\bar{\delta}}$, $w\in\mathbb{W}$: 
\begin{subequations}
\label{eq:robust_MPC_Lyap}
\begin{align}
\label{eq:robust_MPC_Lyap_1}
\bar{\alpha}_1(\|\bar{x}^*_{\tube}(x)\|)\leq& {V}_{\tube}(x)\leq \bar{\alpha}_2(\|\bar{x}^*_{\tube}(x)\|),\\
\label{eq:robust_MPC_Lyap_2}
V_{\tube}(x^+) \leq &V_{\tube}(x)-\ell(\bar{x}^*_{\tube}(x),u),
\end{align}
\end{subequations}
with $x^+=f(x,u)+w$, $u=\pi_{\tube}(x)$.  
If additionally Assumption~\ref{ass:exp_bounds}\ref{ass:exp_bound_MPC} holds, then $\bar{\alpha}_1,\bar{\alpha}_2$ are quadratic.  
\end{theorem}
\begin{pf}
The proof follows standard robust MPC arguments~\cite{bayer2013discrete,singh2023robust}\ifbool{arxiv}{.\\
\textbf{Part I: }
Abbreviate $\bar{x}=\bar{x}_{\tube}^*(x)$ and $\bar{x}^+=f(\bar{x},u)$. 
As a candidate solution for Problem~\eqref{eq:MPC_tube}, we use the initial state $\bar{x}^+$, shift the previous optimal input sequence $\bar{\mathbf{u}}$, and append the terminal control law $k_{\mathrm{f}}$, as standard in MPC. 
Lemma~\ref{lemma:RPI} in combination with $V_\delta(x,\bar{x})\leq \bar{\delta}^2$ and $w\in\mathbb{W}$ ensures satisfaction of the initial state constraint $V_\delta(x^+,\bar{x}^+)\leq \bar{\delta}^2$. 
The candidate solution also satisfies the state, input, and terminal set constraint with standard arguments from nominal MPC~\cite{rawlings2017model,grune2017nonlinear}, which implies recursive feasibility.\\
\textbf{Part II: }The tightened constraints are constructed such that $\bar{x}\in\bar{\mathbb{X}}$ and $V_\delta(x,\bar{x})\leq \bar{\delta}^2$ imply $x\in\mathbb{X}$. Hence, recursive feasibility implies satisfaction of the state constraints.\\ 
\textbf{Part III: }
Inequalities~\eqref{eq:robust_MPC_Lyap} follow with standard arguments from nominal MPC~\cite[Sec.~2.4.2]{rawlings2017model}, see also the proof of Lemma~\ref{lemma:MPC}.}{, see~\cite{JK_Soft_extended} for details.}$\hfill\square$
\end{pf} 
Inequalities~\eqref{eq:robust_MPC_Lyap} ensure that the nominal state $\bar{x}$ converges to the origin and correspondingly the true state $x$ converges to an RPI set around the origin.

\subsection{Proof of Theorem~\ref{thm:RMPC_relax}}
\label{app:robust_proof}
In the following, we proof Theorem~\ref{thm:RMPC_relax}, by using the established properties of the robust MPC (Thm.~\ref{thm:robust_MPC}). 
Note that compared to the robust MPC in Problem~\eqref{eq:MPC_tube}, the proposed formulation in Problem~\eqref{eq:MPC_tube_relax} simply relaxes the initial state constraint with a slack $s$, which is penalized in the cost. 
\begin{pf}
The proof follows the arguments in Theorems~\ref{thm:global}, \ref{thm:constraints}, and \ref{thm:robust_MPC}.\\
\textbf{Part I: }
Given any $x(k)=x\in\mathbb{R}^n$, $w(k)=w\in\mathbb{R}^n$, we denote $\bar{x}=\bar{x}_{\tubeSlack}^*(x)$, $u=\pi_{\tubeSlack}(x)$, $s=\max\left\{\sqrt{V_\delta(x,\bar{x})}-\bar{\delta},0\right\}$, $\bar{x}^+=f(\bar{x},u)$, and $x^+=f(x,u)+w$.
The slack $s^+$ corresponding to the candidate solution $\bar{x}^+$ satisfies
\begin{align*}
s^+:=&\max\left\{\sqrt{V_\delta(x^+,\bar{x}^+)}-\bar{\delta},0\right\}\\
\stackrel{\eqref{eq:RPI}}{\leq}&\rho_\delta\max\{\sqrt{V_\delta(x,\bar{x})}-\bar{\delta},0\}+\sqrt{c_{4,\delta}}\|w\|_{\mathbb{W}}\\ 
=&\rho_\delta s+\sqrt{c_{4,\delta}}\|w\|_{\mathbb{W}}.\nonumber
\end{align*}
Using Cauchy-Schwarz and Young’s inequality with some $\epsilon>0$ yields
\begin{align}
\label{eq:tube_candidate}
(s^+)^2\leq \underbrace{(1+\epsilon)\rho_\delta^2}_{=:\tilde{\rho}_\delta} s^2+\dfrac{1+\epsilon}{\epsilon}c_{4,\delta}\|w\|_{\mathbb{W}}^2,
\end{align}	
with $\tilde{\rho}_\delta<1$ by choosing $\epsilon>0$ small enough.
Analogous to the nominal cost decrease form Theorem~\ref{thm:robust_MPC}, the same candidate solution yields
\begin{align*}
&V_{\tubeSlack}(x^+)-V_{\tubeSlack}(x)\\
\stackrel{\eqref{eq:robust_MPC_Lyap_2}}{\leq}&
\lambda(s^+)^2-\lambda s^2-\ell(\bar{x},u)\\
\stackrel{\mathrm{Asm.}\ref{ass:standing},\ref{ass:exp_bounds}\ref{ass:exp_bound_MPC},\eqref{eq:tube_candidate}}{\leq} &-(1-\tilde{\rho}_\delta)\lambda s^2-c_\ell\|\bar{x}\|^2+\lambda\dfrac{1+\epsilon}{\epsilon}c_{4,\delta}\|w\|_{\mathbb{W}}^2\\
\stackrel{\eqref{eq:robust_MPC_Lyap_1}}{\leq} &-(1-\tilde{\rho})V_{\tubeSlack}(x)+\lambda c_{\mathrm{w}}\|w\|_{\mathbb{W}}^2,
\end{align*}
with $\tilde{\rho}:=\max\{\tilde{\rho}_\delta,\bar{\rho}_{\nom}\}\in[0,1)$, $\bar{\rho}_{\nom}=1-c_\ell/\bar{c}_2\in[0,1)$, $\bar{\alpha}_2(r)=\bar{c}_2r^2$, and $c_{\mathrm{w}}:=\frac{1+\epsilon}{\epsilon}c_{4,\delta}$, i.e.,
Inequality~\eqref{eq:Lyap_robust} holds.\\ 
\textbf{Part II: }
The distance to the constraint set satisfies
\begin{align}
\label{eq:violation_robust_slack_bound}
&\sqrt{c_{\delta,1}}\|x\|_{\mathbb{X}}\leq\sqrt{c_{\delta,1}}\|x\|_{\bar{\mathbb{X}}\oplus\bar{\Omega}_{\bar{\delta}}}\\
\leq&\sqrt{c_{\delta,1}}\|x-\bar{x}\|_{\bar{\Omega}_{\bar{\delta}}}
=\max\{0,\sqrt{c_{\delta,1}}\|x-\bar{x}\|-\bar{\delta}\}\nonumber\\
\leq &\max\left\{\sqrt{V_{\delta}(x,\bar{x})}-\bar{\delta},0\right\}=s\stackrel{\eqref{eq:MPC_tube}}{\leq}\sqrt{V_{\tubeSlack}(x)/\lambda}.\nonumber
\end{align}
Analogous to Theorem~\ref{thm:constraints}, we use Inequality~\eqref{eq:Lyap_robust} in a telescopic sum to obtain
\begin{align*}
&c_{\delta,1}\sum_{k=0}^{T-1}\|x(k)\|_{\mathbb{X}}^2
\stackrel{\eqref{eq:violation_robust_slack_bound}}{\leq} \sum_{k=0}^{T-1} V_{\tubeSlack}(x(k))/\lambda\\
\stackrel{\eqref{eq:Lyap_robust}}{\leq} &\sum_{k=0}^{T-1} \dfrac{\tilde{\rho}^k}{\lambda}V_{\tubeSlack}(x_0)
+c_{\mathrm{w}}\sum_{k=0}^{T-1}\sum_{j=0}^{k-1}\tilde{\rho}^{k-j-1}\|w(j)\|^2_{\mathbb{W}}\\ 
\leq&\dfrac{1}{\lambda(1-\tilde{\rho})}V_{\tubeSlack}(x_0)+\dfrac{c_{\mathrm{w}}}{1-\tilde{\rho}}\sum_{k=0}^{T-1}\|w(k)\|^2_{\mathbb{W}}.
\end{align*} 
Given $\mathcal{J}_N$ continuous, and $\bar{\mathbb{X}}_N,\mathbb{U}$ compact, there exists a uniform upper bound $\bar{V}>0$, such that $\mathcal{J}_N(\bar{x},\mathbf{u})\leq \bar{V}$ for any $\bar{x}\in\bar{\mathbb{X}}_N$, $\mathbf{u}\in\mathbb{U}^N$. 
Consider a candidate $\bar{x}\in\arg\min_{\bar{x}\in \bar{\mathbb{X}}_N}\|x_0-\bar{x}\|$, which satisfies
$\sqrt{V_{\delta}(x_0,\bar{x})}\stackrel{\eqref{eq:increm_Lyap_1}}{\leq} \sqrt{c_{\delta,2}}\|x_0-\bar{x}\|=\sqrt{c_{\delta,2}}\|x_0\|_{\bar{\mathbb{X}}_N}$.
The corresponding slack variable $s$ satisfies
\begin{align*}
s=&\max\{\sqrt{V_{\delta}(x_0,\bar{x})}-\bar{\delta},0\}\\
\leq& \sqrt{c_{\delta,2}}\{\|x_0\|_{\bar{\mathbb{X}}_N}-\bar{\delta}/{\sqrt{c_{\delta,2}}},0\}\\
=&\sqrt{c_{\delta,2}}\min_{\tilde{x}:\|\tilde{x}\|\leq \bar{\delta}/\sqrt{c_{\delta,2}}}\|x_0-\tilde{x}\|_{\bar{\mathbb{X}}_N}=c_{\delta,2}\|x_0\|_{\bar{\mathbb{X}}_N\oplus\underline{\Omega}_{\bar{\delta}}}.
\end{align*}
This implies 
\begin{align*}
V_{\tubeSlack}(x_0)\leq& \bar{V}+\lambda s^2
\leq \bar{V}+\lambda c_{\delta,2}\|x_0\|_{\bar{\mathbb{X}}_N\oplus\underline{\Omega}_{\bar{\delta}}}^2.
\end{align*}
Inequality~\eqref{eq:constraint_robust} follows by choosing $\bar{c}_{\mathbb{X},1}:=\frac{c_{\delta,2}}{(1-\tilde{\rho})c_{\delta,1}}$, $\bar{c}_{\mathbb{X},2}:=\frac{\bar{V}}{(1-\tilde{\rho})c_{\delta,1}}$, and $c_{\mathbb{X},3}:=\frac{c_{\mathrm{w}}}{(1-\tilde{\rho})c_{\delta,1}}$. $\hfill\square$
\end{pf}

\subsection{Discussion}
\label{sec:robust_discussion}
In the standard robust MPC setting (Thm.~\ref{thm:robust_MPC}), we assume initial feasibility of a robust MPC problem ($x_0\in\bar{\mathbb{X}}_N\oplus\underline{\Omega}_{\bar{\delta}}$) and suitable bounds on the disturbances ($w(k)\in\mathbb{W}$). 
In this case, Inequality~\eqref{eq:constraint_robust} ensures that the constraint violation becomes arbitrarily small as we increase the weight $\lambda$. 
Furthermore, Inequality~\eqref{eq:Lyap_robust} ensures that the value function $V_{\tubeSlack}$ exponentially converges to zero and thus the true state $x(k)$ converges to the RPI set $\{x\in\mathbb{R}^n|~V_\delta(x,0)\leq \bar{\delta}^2\}\subseteq\bar{\Omega}_{\bar{\delta}}\subseteq\mathbb{X}$ around the origin. 
Thus, we recover the closed-loop properties of the robust MPC scheme in case $w(k)\in\mathbb{W}$ (cf. Thm.~\ref{thm:robust_MPC}).
Considering overall implementation, compared to the robust MPC~\eqref{eq:MPC_tube}, only a softening of the initial state constraint with a proper weighting is needed. 
This makes the proposed approach very attractive for practical application.

{\textit{Combining robust MPC and soft constraints:}
We note that combinations of soft constraints and robust MPC formulations are, e.g., also suggested in~\cite[Sec.~V.B]{zeilinger2014soft} for linear systems. 
Similar to the derived result, this formulation also recovers the properties of the robust MPC if $w(k)\in\mathbb{W}$ (cf.~\cite[Rk.~V.3]{zeilinger2014soft}). 
However, as with most soft-constrained MPC formulations (cf. Sec.~\ref{sec:linear}), all closed-loop properties in~\cite{zeilinger2014soft} may be lost if (too) large disturbances $w(k)\notin\mathbb{W}$ occur, compare also the example Section~\ref{sec:num_nonlinear}. 
On the other hand, Theorem~\ref{thm:RMPC_relax} ensures that the impact of arbitrarily large disturbances on stability and constraint violations is suitably bound.}

{\textit{Variations in robust MPC design:} 
We have presented the robust MPC formulation according to the problem setting in Section~\ref{sec:theory}: an incremental Lyapunov function $V_\delta$ certifying (exponential) stability is known. 
Following the derivation in Appendix~\ref{app:input}, this can be naturally relaxed to stability under some feedback $\kappa$. In this case, the applied control input $u=\kappa(x,\bar{x},\bar{u})$ steers the state $x$ to the nominal state $\bar{x}$, as standard in robust the MPC~\cite{singh2023robust}.
In case of a large disturbances $w\notin\mathbb{W}$, this combination results in violations of the input constraints. 
Knowledge of an explicit Lyapunov function $V_\delta$ can also be relaxed with an implicit trajectory-based characterization, see Appendix~\ref{app:implicit}. 
This is also of relevance for the design of nonlinear robust MPC schemes in general.}

\section{Details regarding the numerical examples}
\label{app:num}
In the following, we provide additional details for the numerical examples in Section~\ref{sec:num}.
All computations were carried out using Matlab, on a Acer Aspire A715-75G laptop with Intel i7-10750H CPU and 32.0 GB RAM running Windows 10. 
The optimization problems are solved using quadprog and IPOPT~\cite{wachter2006implementation} with CasADi~\cite{andersson2019casadi}. 
\subsection{Comparison to linear soft-constrained MPC}
We consider a mass-spring-damper system taken from~\cite{wabersich2021soft} with spring constant $k=1$, mass $m=1$, damping factor $d=0.1$, resulting in the continuous-time dynamic system: 
\begin{align*}
\dot{x}=\begin{pmatrix}
0&1\\-k/m & -d/m
\end{pmatrix}x+\begin{pmatrix}
0\\1
\end{pmatrix}u.
\end{align*}
The sampling time is chosen as $0.05$ and the discrete-time matrices $A,B$ are computed using the MATLAB command \textit{c2d}. 
The quadratic stage cost is given by $Q=\text{diag}(1,0.1)$, $R=0.2$ and the prediction horizon is $N=10$. 
The quadratic penalties for the soft constraints (cf. Sec.~\ref{sec:linear}) are weighted using $\lambda=10^3$, $Q_\xi=10^5 I_p$, resulting in negligible constraint violations on the feasible set $\mathbb{X}_N$.

We compute the terminal cost $V_{\mathrm{f}}$ using the LQR and $\mathbb{X}_{\mathrm{f}}$ as the maximal positive invariant set. 
Given the nominal MPC, the offline design and implementation of the proposed formulation required two additional lines of Matlab code to compute $V_\delta$ using \textit{dLyap} and replacing the initial state constraint by the quadratic penalty.
The soft-\textbf{P} formulation simply relaxes the state constraint with the quadratic penalty $\|\xi\|_{Q_\xi}^2$. 
The soft-\textbf{T} formulation computes a larger terminal set $\mathbb{X}_{\mathrm{f}}$ with soft state constraints using the code from the paper~\cite{wabersich2021soft}. 
The soft-\textbf{G} formulation computes a separate terminal cost $V_{\mathrm{f}}$ using the \textit{dLyap} command (cf. Sec.~\ref{sec:linear}).  
All conditions posed in the main theoretical results (Sec.~\ref{sec:theory}) hold for this linear example (cf. Section~\ref{sec:linear}).  
The average computation times are computed by gridding the state constraint $\mathbb{X}\subseteq\mathbb{R}^2$ with $101^2$ points and excluding any point where the nominal MPC is not feasible. 

\subsection{Nonlinear system subject to disturbances}
The dynamics of the four-tank system~\cite{limon2018nonlinear} can be compactly written as:
\begin{align*}
\dot{x}_1=&-c_1\sqrt{x_1}+c_{2}\sqrt{x_2}+c_{1,\mathrm{u}}u_1,~~ 
\dot{x}_2=-c_2\sqrt{x_2}+c_{2,\mathrm{u}}u_2, \\
\dot{x}_3=&-c_3\sqrt{x_3}+c_{4}\sqrt{x_4}+c_{3,\mathrm{u}}u_2,\nonumber~~ 
\dot{x}_4=-c_4\sqrt{x_4}+c_{4,\mathrm{u}}u_1,\nonumber
\end{align*}
with states $x\in\mathbb{R}_{\geq 0}^4$, inputs $u\in\mathbb{R}^2_{\geq 0}$. 
The constants are given by $c_1=0.0096$, $c_2=0.0111$, $c_3=0.0075$, $c_4=0.0069$, $c_{1,\mathrm{u}}=0.0014$, $c_{2,\mathrm{u}}=0.0019$, $c_{3,\mathrm{u}}=0.0028$, $c_{4,\mathrm{u}}=0.0032$.
The constraints are given by $u\in\mathbb{U}=\{u\in\mathbb{R}^2_{\geq 0}|~u_1\leq 3.6,~u_2\leq 4\}$ and $\mathbb{X}=\{x|~x_i\geq 0.2,~i\in\mathbb{I}_{[1,4]},~x_i\leq 1.36, i=1,3~x_i\leq 1.4, i=2,4\}$.

Next, we derive a  quadratic  incremental Lyapunov function $V_\delta$ $V_\delta(x,z)=\|x-z\|_P^2$ certifying open-loop stability (Asm.~\ref{ass:increm_stab}). 
It suffices to ensure $A^\top(x) P +PA(x)\prec 0$ (cf., e.g.,~\cite{tran2019convergence})
with the (continuous-time) Jacobian/linearization
\begin{align*}
A(x)=-0.5
\begin{pmatrix}
\tilde{c}_1&-\tilde{c}_2&0&0\\
0&\tilde{c}_2&0&0\\
0&0&\tilde{c}_3&-\tilde{c}_4\\
0&0&0&\tilde{c}_4
\end{pmatrix},\quad \tilde{c}_i=c_ix_i^{-0.5}. 
\end{align*} 
Given the structure of the open-loop system, we consider the diagonal matrix $P=\text{diag}(1,p_2,1,p_4)$ and the stability condition reduces to
\begin{align*} 
\begin{pmatrix}
-2\tilde{c}_1&\tilde{c}_2\\
\tilde{c}_2&-2\tilde{c}_2p_2\\
\end{pmatrix}\prec 0,\quad 
\begin{pmatrix}
-2\tilde{c}_3&\tilde{c}_4\\
\tilde{c}_4&-2\tilde{c}_4p_4
\end{pmatrix}\prec 0. 
\end{align*}
The physical model is only well-defined for $x_i\in(0,\infty)$ and hence open-loop stability (Asm.~\ref{ass:increm_stab}) and closed-loop properties cannot be shown globally for all $x\in\mathbb{R}^n$.
Instead, we provide semi-global results in the sense that for each compact interval $x_i\in[\underline{x},\overline{x}]$, $i\in\mathbb{I}_{[1,4]}$, $0<\underline{x}\leq\overline{x}<\infty$, we can compute a suitable matrix $P$.
We use $\overline{x}=2$, $\underline{x}=0.02$, resulting in $P=\text{diag}(1,2,1,2)$. 
All simulations are contained in this interval and Assumption~\ref{ass:increm_stab} holds with $V_\delta(x,z)=\|x-z\|_P^2$ on the considered interval (neglecting discretization error). 

The nominal MPC scheme is implemented using an additional artificial setpoint $y_{\mathrm{s}}$ for the terminal constraints to track the desired output $y_{\mathrm{d}}$~\cite{limon2018nonlinear}. 
This enlarges the feasible set $\mathbb{X}_N$, which is also exploited in the soft-constrained MPC in~\cite{zeilinger2014soft}.
The terminal set $\mathbb{X}_{\mathrm{f}}$ is a simple terminal equality constraint w.r.t. the artificial setpoint.
This also ensures recursive feasibility of the nominal MPC in the absence of disturbances (cf.~\cite{limon2018nonlinear} for details). 
We use an explicit fourth order Runge Kutta discretization with sampling time of $10$s. 
The prediction horizon is $N=10$.
We consider quadratic penalties with $V_{\mathrm{o}}=10^4\|y_{\mathrm{s}}-y_{\mathrm{t}}\|^2$ (cf.~\cite{limon2018nonlinear}), $Q_\xi=10^4\cdot I_p$, $\lambda=10^5$, $Q=I_n$, $R=10^{-2}I_m$. 
The level set of the RPI set for the robust approach is chosen as $\bar{\delta}=0.05$.

}{}
\end{document}